\documentclass[12pt]{amsart}
\usepackage{latexsym,amsmath,amssymb,epsfig,amsthm}
\usepackage{graphicx}
\usepackage{caption}
\usepackage{subcaption}
\usepackage{tikz}
\usepackage[enableskew,vcentermath]{youngtab}
\usepackage{hyperref}
\usepackage{cleveref}
\hypersetup{colorlinks=false}
\input epsf
\newtheorem{theorem}{Theorem}[section]
\newtheorem{proposition}[theorem]{Proposition}
\newtheorem{corollary}[theorem]{Corollary}
\newtheorem{conjecture}[theorem]{Conjecture}
\newtheorem{definition}[theorem]{Definition}
\newtheorem{example}[theorem]{Example}
\newtheorem{lemma}[theorem]{Lemma}
\newtheorem{remark}[theorem]{Remark}

\newcommand{\depth}{{\rm depth}}

\newcommand{\exc}{{\rm exc}}

\newcommand{\link}{{\rm link}}

\newcommand{\reg}{{\rm reg}}
\newcommand{\sd}{{\rm sd}}

\newcommand{\Shift}{{\rm shift}}

\newcommand{\Star}{{\rm star}}

\newcommand{\Tor}{{\rm Tor}}
\newcommand{\aA}{{\mathcal A}}

\newcommand{\cC}{{\mathcal C}}

\newcommand{\fF}{{\mathcal F}}
\newcommand{\gG}{{\mathcal G}}

\newcommand{\qQ}{{\mathcal Q}}

\newcommand{\RR}{{\mathbb R}}
\newcommand{\fD}{{\mathfrak D}}
\newcommand{\fS}{{\mathfrak S}}
\newcommand{\NN}{{\mathbb N}}

\newcommand{\FF}{{\mathbb F}}
\renewcommand{\to}{\rightarrow}

\newcommand{\sm}{{\smallsetminus}}
\begin{document}
\title[Combinatorics of antiprism triangulations]
{Combinatorics of antiprism triangulations}

\author{Christos~A.~Athanasiadis}
\address{Department of Mathematics\\
National and Kapodistrian University of Athens\\
Panepistimioupolis\\
15784 Athens, Greece}
\email{caath@math.uoa.gr}

\author{Jan-Marten~Brunink}
\address{Universit\"at Osnabr\"uck\\
Fakult\"at f\"ur Mathematik\\
Albrechtstrasse 28A\\
49076 Osnabr\"uck, Germany}
\email{janmarten.brunink@uni-osnabrueck.de}

\author{Martina~Juhnke-Kubitzke}
\address{Universit\"at Osnabr\"uck\\
Fakult\"at f\"ur Mathematik\\
Albrechtstrasse 28A\\
49076 Osnabr\"uck, Germany}
\email{juhnke-kubitzke@uni-osnabrueck.de}

\date{August 18, 2021}

\begin{abstract}
The antiprism triangulation provides a natural way 
to subdivide a simplicial complex $\Delta$, similar to 
barycentric subdivision, which appeared independently 
in combinatorial algebraic topology and computer 
science. It can be defined as the simplicial complex 
of chains of multi-pointed faces of $\Delta$, from 
a combinatorial point of view, and by successively 
applying the antiprism construction, or balanced 
stellar subdivisions, on the faces of $\Delta$, from 
a geometric point of view. 

This paper studies enumerative invariants associated 
to this triangulation, such as the transformation of 
the $h$-vector of $\Delta$ under antiprism 
triangulation, and algebraic properties of its 
Stanley--Reisner ring. Among other results, it is 
shown that the $h$-polynomial of the antiprism 
triangulation of a simplex is real-rooted and that 
the antiprism triangulation of $\Delta$ has the 
almost strong Lefschetz property over $\RR$ for every 
shellable complex $\Delta$. Several related open 
problems are discussed.
\end{abstract}

\maketitle

\section{Introduction}
\label{sec:intro}

Barycentric subdivision provides a natural way to 
triangulate a simplicial complex $\Delta$, of 
fundamental importance in algebraic topology. Because 
of its especially nice enumerative and algebraic 
properties, it has also been studied intensely from 
the point of view of enumerative and algebraic 
combinatorics \cite{BW08, CKW18, KMS19, KN09, 
KW08, Mur10b, NPT11}, \cite[Chapter~9]{Pet15}. For 
instance, Brenti and Welker~\cite{BW08} described in 
explicit combinatorial terms the transformation of the 
$h$-vector (a fundamental enumerative invariant) of 
$\Delta$, under barycentric subdivision, and showed 
that the $h$-polynomial (the generating polynomial 
for the $h$-vector) of the barycentric subdivision 
of $\Delta$ has only real roots (and in particular, 
log-concave and unimodal coefficients) for every 
simplicial complex $\Delta$ with nonnegative 
$h$-vector. 

A similar, but combinatorially more intricate and 
much less studied than barycentric subdivision, 
way to subdivide $\Delta$ is provided by the 
\emph{antiprism triangulation}, denoted here by 
$\sd_\aA(\Delta)$. To give the reader a hint on 
the comparison between the two triangulations, we 
recall that the barycentric subdivision of a 
geometric simplex $\Sigma$ can be constructed 
by inserting a vertex in the interior of $\Sigma$ 
and coning over its proper faces, which have been 
barycentrically subdivided by induction. The 
antiprism triangulation $\sd_\aA(\Sigma)$ instead 
can be constructed by inserting another simplex of
the same dimension in the interior of $\Sigma$, 
whose vertices are in a given one-to-one 
correspondence with those of $\Sigma$, and 
joining each nonempty face of that simplex with
the antiprism triangulation of the complementary 
face of $\Sigma$. Figure~\ref{fig:2simplex} shows 
the antiprism triangulation of a 2-dimensional 
simplex (the labeling of faces is explained in 
Section~\ref{sec:atriang}). As an abstract 
simplicial complex, the barycentric subdivision 
of $\Delta$, denoted here by $\sd(\Delta)$,
has faces which correspond bijectively to the 
ordered partitions of the faces of $\Delta$; in 
particular, the vertices and facets of $\sd(\Delta)$ 
correspond bijectively to the nonempty faces and the 
permutations of the facets of $\Delta$, respectively. 
The faces of $\sd_\aA(\Delta)$ instead correspond 
bijectively to certain multi-pointed ordered 
partitions of the faces of $\Delta$; in particular, 
the vertices and facets of $\sd_\aA(\Delta)$ 
correspond bijectively to the pointed faces and the 
ordered partitions of the facets of $\Delta$, 
respectively.

The antiprism triangulation was introduced 
by Izmestiev and Joswig \cite{IJ03} as a technical 
device in their effort to understand combinatorially 
branched coverings of manifolds, and arose 
independently and was studied under the name 
\emph{chromatic subdivision} in computer science 
(specifically, in theoretical distributed computing); 
see \cite{Ko12} and references therein.
This paper aims to show that, as is the case with 
barycentric subdivision, the antiprism
triangulation has very interesting enumerative and 
algebraic properties and that its study leads to 
combinatorial problems which are often more 
challenging than the corresponding ones for the 
barycentric subdivision. We denote by 
$h(\Delta, x)$ the $h$-polynomial of a simplicial 
complex $\Delta$ and by $\sigma_n$ the (abstract) 
simplex on an $n$-element vertex set. Our main 
motivation comes from the following conjectural 
analogue of the main result of~\cite{BW08}.

\begin{figure}
\centering
  \includegraphics[width=0.7\textwidth]{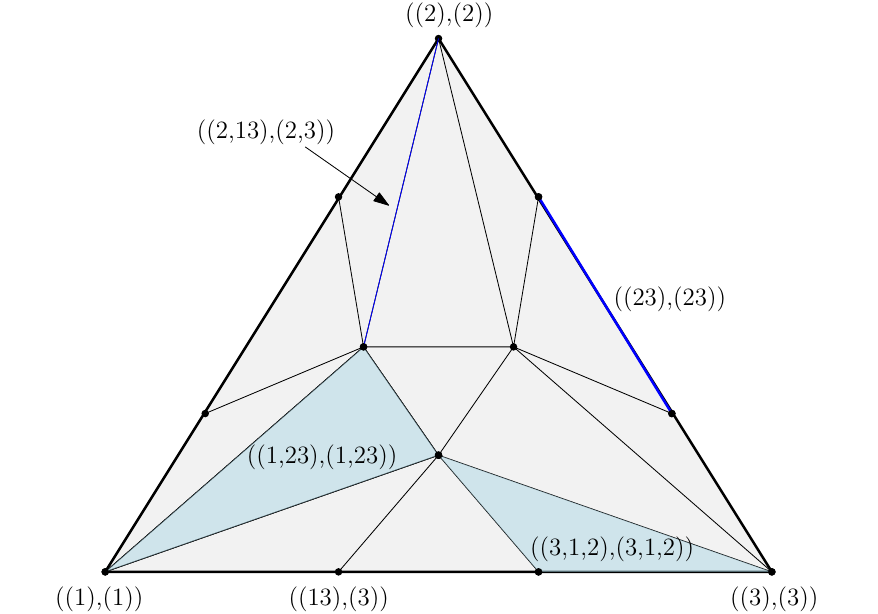}
  \caption{Antiprism triangulation of the 2-simplex}
  \label{fig:2simplex}
\end{figure}

\begin{conjecture} \label{conj:main}
The polynomial $h(\sd_\aA(\Delta), x)$ is real-rooted
for every simplicial complex $\Delta$ with nonnegative 
$h$-vector. 
\end{conjecture}

This conjecture is part of the general problem to 
understand when the $h$-polynomial of a triangulation 
of a simplicial complex is real-rooted. The present 
study of antiprism triangulations has partly motivated 
the study of this problem for the much more general 
class of uniform triangulations~\cite{Ath20+}. Although 
we are unable to fully settle \Cref{conj:main} in this 
paper, we reduce it to an interlacing relation
between the members of two concrete infinite sequences
of polynomials (see Conjecture~\ref{conj:thetaA}), 
given the following important special case of the 
conjecture and \cite[Theorem~1.2]{Ath20+}.
\begin{theorem} \label{thm:mainA}
The polynomial $h(\sd_\aA(\sigma_n), x)$ is real-rooted
and has a nonnegative, real-rooted and interlacing 
symmetric decomposition with respect to $n-1$ for 
every positive integer $n$. 
\end{theorem}

We also prove the unimodality of $h(\sd_\aA(\Delta), x)$ 
for every Cohen--Macaulay simplicial complex $\Delta$
and show that the peak appears in the middle, by 
studying Lefschetz properties of the Stanley--Reisner 
ring of $\sd_\aA(\Delta)$. The following result is 
an analogue of the main result of~\cite{KN09} for the
barycentric subdivision. 
\begin{theorem} \label{thm:LefschetzUnimodal}
The complex $\sd_\aA(\Delta)$ has the almost strong 
Lefschetz property over $\RR$ for every shellable 
simplicial complex $\Delta$. 

Moreover, for every $(n-1)$-dimensional 
Cohen--Macaulay simplicial complex $\Delta$, the 
$h$-vector of $\sd_\aA(\Delta)$ is unimodal, with 
the peak being at position $n/2$, if $n$ is even, 
and at $(n-1)/2$ or $(n+1)/2$, if $n$ is odd. 
\end{theorem}

This paper is structured as follows. The antiprism 
triangulation $\sd_\aA(\Delta)$ is described 
combinatorially as an abstract simplicial complex 
and defined geometrically as a triangulation, using
either the antiprism construction, or balanced 
stellar subdivisions (crossing operations), in 
Section~\ref{sec:atriang}. The antiprism construction 
is defined in Section~\ref{sec:aconstr}, where its 
face enumeration is studied within the framework of 
uniform triangulations, introduced in~\cite{Ath20+}. 
These results are then applied in 
Section~\ref{sec:enumerate} to find combinatorial 
interpretations and recurrences for the basic 
enumerative invariants of the antiprism triangulation 
of the simplex. The face enumeration of antiprism 
triangulations turns out to be related to traditional 
combinatorial themes, such as ordered set partitions,
colorings and the enumeration of permutations by 
excedances (for example, the number of facets of 
$\sd_\aA(\sigma_n)$ is equal to the number of ordered 
partitions of an $n$-element set). 
Section~\ref{sec:enumerate} also proves 
Theorem~\ref{thm:mainA} and describes combinatorially 
the transformation of the $h$-vector of a simplicial 
complex, under antiprism triangulation. The proof of
Theorem~\ref{thm:mainA} is different from all proofs 
of the corresponding result for the barycentric subdivision
known to the authors; it exploits the recurrence 
for the $h$-polynomial of $\sd_\aA(\sigma_n)$ and 
uses the concept of interlacing sequence of 
polynomials. Theorem~\ref{thm:LefschetzUnimodal} is 
proved in Section~\ref{sec:aLefschetz}; the method 
generally follows those of \cite{KN09, Mur10}, with 
certain complications and shortcuts.

Basic background and definitions, together with 
some preliminary technical results, are included
in Section~\ref{sec:pre} for simplicial complexes,
their triangulations and face enumeration, and for 
the unimodality and real-rootedness of polynomials
and their symmetric decompositions, and in 
Section~\ref{sec:Lefschetz} 
for Lefschetz properties of simplicial complexes. 
Open problems, other than those proposed earlier
in the paper, and further directions for research
are discussed in Section~\ref{sec:direct}.

\section{Preliminaries}
\label{sec:pre}

This section includes preliminaries on simplicial 
complexes and triangulations, their basic enumerative 
invariants and the unimodality of polynomials and 
related properties. Throughout this paper we set 
$\NN := \{0, 1, 2,\dots\}$ and $[n] := 
\{1, 2,\dots,n\}$ for $n \in \NN$. We also denote 
by $\fS_n$ the symmetric group of permutations of 
$[n]$ and by $|V|$ and $2^V$ the cardinality and 
the powerset, respectively, of a finite set $V$. 

\subsection{Simplicial complexes}
\label{sec:complexes}

We start with several definitions and refer to Stanley's 
book \cite{StaCCA} for background and more information. 

Let $V$ be a finite set. An (abstract) \emph{simplicial 
complex} $\Delta$ on the vertex set $V$ 
is a collection of subsets of $V$ that is closed under 
inclusion and such that $\{v\} \in \Delta$ for every
$v\in V$. Throughout this article, we assume that all 
simplicial complexes are finite. The elements of 
$\Delta$ are called \emph{faces} and the inclusionwise 
maximal ones are called \emph{facets}. The 
\emph{dimension} of a face $F \in \Delta$ is defined 
as $\dim(F) = |F|-1$; the \emph{dimension} of $\Delta$, 
denoted by $\dim(\Delta)$, is the maximum dimension of 
its faces. Zero-dimensional and one-dimensional faces 
of $\Delta$ are called \emph{vertices} and \emph{edges}, 
respectively. We say that $\Delta$ is \emph{pure} if 
all facets of $\Delta$ have the same dimension. As 
in~\cite{Ath20+}, we denote by $\sigma_n$ the abstract 
$(n-1)$-dimensional simplex $2^V$ on an $n$-element 
vertex set $V$ (often taken to be $[n]$). 

The \emph{cone} over $\Delta$ is the simplicial 
complex consisting of the faces of $\Delta$, together 
with all sets $F \cup \{u\}$ for $F \in \Delta$, where 
$u \notin V$ is a new vertex, called the \emph{apex}. 
We will denote this cone by $u \ast \Delta$. More 
generally, the (simplicial) \emph{join} of two 
simplicial complexes $\Delta_1$ and $\Delta_2$ with 
disjoint vertex sets is defined as $\Delta_1 \ast 
\Delta_2 = \{ F_1 \cup F_2~ :~ F_1 \in \Delta_1, F_2 
\in \Delta_2\}$. Given a face $F \in \Delta$, the 
\emph{link} and the \emph{star} of $F$ in $\Delta$ 
are defined as the simplicial complexes
\begin{center}
\begin{tabular}{l}
$\link_\Delta(F) \ = \ \{ G \in \Delta~:~ F \cup 
G \in \Delta,\ F \cap G = \varnothing \}$ and \\ 
$\Star_\Delta(F) \ = \ \{ G \in \Delta~: ~F \cup 
G \in \Delta\}$,
\end{tabular} 
\end{center}
respectively. For $G_1, G_2,\dots,G_m \subseteq V$ 
we set
\begin{equation*}
\langle G_1, G_2,\dots,G_m \rangle \ = \ 
\{F~:~ F \subseteq G_i \text{ for some } i \in [m]
       \}.
\end{equation*}

In the sequel, $\Delta$ is a pure $(n-1)$-dimensional 
simplicial complex with vertex set $V$ and $\FF$ is a 
field. Let $A$ be the polynomial ring $\FF[x_v: v 
\in V]$ and write $x_F = \prod_{v \in F} x_v$ for $F 
\subseteq V$. 
The \emph{Stanley-Reisner ring} (or \emph{face ring}) 
of $\Delta$ (over $\FF$) is defined as the quotient 
ring $\FF[\Delta] = A/I_\Delta$, where $I_\Delta=
(x_F: F \subseteq V,\ F \not \in \Delta)$ is the 
ideal of $A$ known as the \emph{Stanley-Reisner ideal} 
(or \emph{face ideal}) of $\Delta$. The ring 
$\FF[\Delta]$ is graded by degree; subscripts on 
$\FF[\Delta]$ and its (standard) graded quotients 
will always refer to homogeneous components.

A \emph{linear system of parameters} (l.s.o.p. for 
short) for $\FF[\Delta]$ is a sequence 
$\Theta=\theta_1,\dots,\theta_n$ of linear forms in 
$\FF[\Delta]$ such that the quotient $\FF[\Delta] / 
\Theta\FF[\Delta]$ has finite dimension, as a vector 
space over $\FF$. The complex $\Delta$ is called 
\emph{Cohen--Macaulay} over $\FF$ if $\FF[\Delta]$
is a free module over the polynomial ring 
$\FF[\Theta]$ for some (equivalently, for every) 
l.s.o.p. $\Theta$ for $\FF[\Delta]$ and 
\emph{shellable} if there exists a linear ordering 
$G_1, G_2,\dots,G_m$ of the facets of $\Delta$ such 
that for each $2 \le j \le m$, the set 
\begin{equation*}
\{F \subseteq G_j~:~ F \not\subseteq G_i \text{ for } 
  1 \le i < j\}
\end{equation*}
has a unique minimal element, with respect to inclusion. 
Even though shellable simplicial complexes constitute a 
proper subclass of that of Cohen--Macaulay complexes, 
the sets of possible $f$-vectors for the two classes 
of simplicial complexes coincide (see, e.g., 
\cite[Theorem~3.3]{StaCCA}).

Given an $(n-1)$-dimensional simplicial complex 
$\Delta$, the \emph{$f$-vector} of $\Delta$ is defined
as the sequence $f(\Delta) = (f_{-1}(\Delta), 
f_0(\Delta),\dots,f_{n-1}(\Delta))$, where $f_i
(\Delta)$ denotes the number of $i$-dimensional faces 
of $\Delta$. The \emph{$h$-vector} of $\Delta$ is 
defined as $h(\Delta)=(h_0(\Delta), 
h_1(\Delta),\dots,h_n(\Delta))$, where $h_i(\Delta)$ 
is given by the formula
\begin{equation*}
h_i(\Delta) \ = \ \sum_{j=0}^i(-1)^{i-j}{n-j \choose 
             i-j} f_{j-1}(\Delta),
\end{equation*}
and $h(\Delta,x) = \sum_{i=0}^n h_i(\Delta)x^i$ is the 
\emph{$h$-polynomial} of $\Delta$. Equivalently, the 
latter can be defined by the formula 
\begin{equation}
\label{eq:hdef}
h(\Delta, x) \ = \ \sum_{i=0}^n f_{i-1} (\Delta) \, 
x^i (1-x)^{n-i} \ = \ \sum_{F \in \Delta} x^{|F|} 
(1-x)^{n-|F|}.
\end{equation}
Assume now that $\Delta$ triangulates an 
$(n-1)$-dimensional ball, meaning that the geometric 
realization of $\Delta$ is homeomorphic to an 
$(n-1)$-dimensional ball (we also say that $\Delta$ 
is a an $(n-1)$-dimensional simplicial ball). The 
\emph{boundary complex} of $\Delta$ is then defined 
as
\begin{equation*}
\partial\Delta \ = \ \langle F \in \Delta~:~F\subseteq 
G \text{ for a unique facet } G \in \Delta\rangle.
\end{equation*}
The set $\Delta^\circ = \Delta \sm \partial\Delta$ 
consists of the \emph{interior} faces of $\Delta$
and $h^\circ(\Delta,x)$ is defined by the 
sum on the far right of (\ref{eq:hdef}) in which 
$\Delta$ has been replaced by $\Delta^\circ$. The 
following well known statement is a special case of 
\cite[Lemma~6.2]{Sta87}.
\begin{proposition} \label{prop:h-ballcirc}
{\rm (\cite{Sta87})} We have $x^n h (\Delta, 
1/x) = h^\circ(\Delta, x)$ for every 
triangulation $\Delta$ of an $(n-1)$-dimensional 
ball.
\end{proposition}

\subsection{Triangulations}
\label{sec:triang}

Let $\Delta$ and $\Delta'$ be simplicial complexes. 
We say that $\Delta'$ is a \emph{triangulation} of 
$\Delta$ if there exist geometric realizations $K'$ 
and $K$ of $\Delta'$ and $\Delta$, respectively, such 
that $K'$ geometrically subdivides $K$. Let $L \in K$ 
be a simplex and $F$ be the corresponding face of 
$\Delta$. Then, $K'$ restricts to a triangulation 
$K'_L$ of $L$. The subcomplex $\Delta'_F$ of $\Delta'$ 
which corresponds to $K'_L$ is a triangulation of the 
abstract simplex $2^F$, called the \emph{restriction} 
of $\Delta'$ to $F$. The \emph{carrier} of a face 
$G \in \Delta'$ is the smallest face $F \in \Delta$ 
such that $G \in \Delta'_F$. 

A fundamental enumerative invariant of a 
triangulation of a simplex is the \emph{local 
$h$-polynomial}. Given a triangulation $\Gamma$ of an 
$(n-1)$-dimensional simplex $2^V$, this polynomial is 
defined \cite[Definition~2.1]{Sta92} by the formula 
\begin{equation} \label{eq:deflocalh}
  \ell_V (\Gamma, x) \ = \sum_{F \subseteq V} 
  \, (-1)^{n - |F|} \, h (\Gamma_F, x).
\end{equation}
By the principle of inclusion-exclusion, 
\begin{equation} \label{eq:h-localh}
  h (\Gamma, x) \ = \sum_{F \subseteq V} 
  \ell_F (\Gamma_F, x).
\end{equation}

Stanley~\cite{Sta92} showed 
that $\ell_V (\Gamma, x)$ has nonnegative and
symmetric coefficients, so that $x^n 
\ell_V (\Gamma, 1/x) = \ell_V (\Gamma, x)$, for 
every triangulation $\Gamma$ of $2^V$, and that it 
has unimodal coefficients for every \emph{regular} 
triangulation, meaning that $\Gamma$ can be 
realized as the collection of projections on a 
geometric simplex of the lower faces of a simplicial 
polytope of one dimension higher.

The \emph{barycentric subdivision} of 
a simplicial complex $\Delta$ is defined as the 
simplicial complex $\sd(\Delta)$ on the vertex set 
$\Delta \sm \{\varnothing\}$ whose faces are the chains 
$F_0 \subsetneq F_1 \subsetneq \cdots \subsetneq F_k$ 
of nonempty faces of $\Delta$. The carrier of such a 
chain is its top element $F_k$. 
To describe the $h$-polynomial and local 
$h$-polynomial of $\sd(\sigma_n)$, we need to recall
a few definitions from permutation enumeration. An 
\emph{excedance} of a permutation $w \in \fS_n$ is 
an index $i \in [n-1]$ such that $w(i) > i$. Let 
$\exc(w)$ be the number of excedances of $w$. 
The polynomial
\begin{equation*}
A_n(x) \ = \ \sum_{w \in \fS_n} x^{\exc(w)}
\end{equation*}
is called \emph{$n$th Eulerian polynomial}	$A_n(x)$; 
see \cite[Section~1.4]{StaEC1} for more information
on this important concept. Similarly, the \emph{$n$th 
derangement polynomial} is defined by the formula
\begin{equation*}
d_n(x) \ = \ \sum_{w \in \fD_n}x^{\exc(w)},
\end{equation*}
where $\fD_n$ the set of all derangements (permutations
without fixed points) in $\fS_n$. Then, $h(\sd(\sigma_n), 
x) = A_n(x)$ and $\ell_V(\sd(\sigma_n), x) = d_n(x)$ for 
every $n$ (see \cite[Section~2]{Sta92}), where $V$ is the 
vertex set of $\sigma_n$.

Let $\fF = (f_\fF(i,j))$ be a triangular array of 
nonnegative integers, defined for $0 \le i \le j$. A 
triangulation $\Delta'$ of a simplicial complex $\Delta$ 
is called \emph{$\fF$-uniform} if for every 
$(n-1)$-dimensional face $F \in \Delta$, the restriction 
$\Delta'_F$ has exactly $f_\fF(k,n)$ faces of 
dimension $k-1$ for all $0 \le k \le n$. The 
barycentric subdivision is a prototypical example of an 
$\fF$-uniform triangulation, for a suitable array $\fF$;
the antiprism triangulation is another. The class of 
$\fF$-uniform triangulations was introduced and studied 
in~\cite{Ath20+}. The $h$-polynomial and local 
$h$-polynomial of an $\fF$-uniform triangulation of 
an $(n-1)$-dimensional simplex depend only on $\fF$ and
$n$ and will be denoted by $h_\fF(\sigma_n, x)$ and 
$\ell_\fF(\sigma_n, x)$, respectively.

\subsection{Polynomials}
\label{sec:polys}

We recall some basic definitions and useful facts about 
unimodal and real-rooted polynomials. A polynomial $p(x) 
= a_0 + a_1 x + \cdots + a_n x^n \in \RR[x]$ is called
\begin{itemize}
\item[$\bullet$] 
  \emph{symmetric}, with center of symmetry $n/2$, if 
	$a_i = a_{n-i}$ for all $0 \le i \le n$,
\item[$\bullet$] 
  \emph{unimodal}, with a peak at position $k$, if $a_0 \le 
	 a_1 \le \cdots \le a_k \ge a_{k+1} \ge \cdots \ge a_n$,
\item[$\bullet$] 
  \emph{alternatingly increasing} with respect to $n$, if 
	 $a_0 \le a_n \le	a_1 \le a_{n-1} \le \cdots \le 
	 a_{\lceil n/2 \rceil}$,
\item[$\bullet$] 
  \emph{$\gamma$-positive}, with center of symmetry $n/2$,
	 if $p(x) = \sum_{j=0}^{\lfloor n/2 \rfloor} \gamma_j 
	 x^j (1+x)^{n-2j}$ for some nonnegative real numbers 
	 $\gamma_0, \gamma_1,\dots,\gamma_{\lfloor n/2 \rfloor}$.
\end{itemize}
Gamma-positivity implies palindromicity and unimodality; 
see \cite{Ath18} for a survey about this very 
interesting concept.	

A polynomial $p(x) \in \RR[x]$ is \emph{real-rooted} 
if all complex roots of $p(x)$ are real, or $p(x)$ is 
the zero polynomial. A real-rooted polynomial, with 
roots $\alpha_1 \ge \alpha_2 \ge \cdots$, is said to 
\emph{interlace} another real-rooted polynomial, with 
roots $\beta_1 \ge \beta_2 \ge \cdots$, if
\[ \cdots \le \alpha_2 \le \beta_2 \le \alpha_1 \le
   \beta_1. \]
By convention, the zero polynomial interlaces and is 
interlaced by every real-rooted polynomial and constant 
polynomials interlace all polynomials of degree at most 
one. Background on real-rooted polynomials and the theory 
of interlacing can be found in \cite{Bra15, Fi06, Sta89} 
and references therein. We recall here the crucial facts 
that every real-rooted polynomial with nonnegative 
coefficients is unimodal and that (see 
\cite[Lemma~3.4]{Fi06}) if two real-rooted polynomials 
$p(x)$ and $q(x)$ have positive leading coefficients and 
$p(x)$ interlaces $q(x)$, then $p(x) + q(x)$ is 
real-rooted as well and it is interlaced by $p(x)$ and 
interlaces $q(x)$. Moreover, every symmetric real-rooted 
polynomial with nonnegative coefficients is 
$\gamma$-positive.

A sequence $(p_0(x), p_1(x),\dots,p_m(x))$ of 
real-rooted polynomials is called \emph{interlacing} if 
$p_i(x)$ interlaces $p_j(x)$ for $0 \le i < j \le m$. 
The following lemma will be used for the proof of 
Theorem~\ref{thm:mainA} in Section~\ref{sec:enu-simplex}.
\begin{lemma} \label{lem:inter}
\begin{itemize}
\itemsep=0pt
\item[(a)]
{\rm (\cite[Lemma~2.3]{Bra06}, 
\cite[Proposition~3.3]{Wa92})}
Let $p_1(x), p_2(x),\dots,p_m(x)$ be real-rooted 
polynomials in $\RR[x]$. If $p_1(x)$ interlaces 
$p_m(x)$ and $p_i(x)$ interlaces $p_{i+1}(x)$ for 
all $i \in [m-1]$, then $(p_1(x), p_2(x),\dots,p_m(x))$ 
is an interlacing sequence.

\item[(b)]
{\rm (cf. \cite[Lemma~3.4]{Fi06})}
If $(p_1(x), p_2(x),\dots,p_m(x))$ is an interlacing 
sequence of real-rooted polynomials in $\RR[x]$ with 
positive leading coefficients, then so is $(p_1(x) + 
p_2(x) + \cdots + p_m(x),\dots,p_{m-1}(x) + p_m(x), 
p_m(x))$.

\item[(c)]
Let $(p_1(x), p_2(x),\dots,p_m(x))$ be an interlacing 
sequence of real-rooted polynomials in $\RR[x]$ with 
positive leading coefficients. Then, $p_1(x) + p_2(x) 
+ \cdots + p_m(x)$ interlaces $c_1 p_1(x) + c_2 p_2(x) 
+ \cdots + c_m p_m(x)$ for all positive real numbers 
$c_1 \le c_2 \le \cdots \le c_m$.
In particular, $p_1(x) + p_2(x) + \cdots + p_{m-1}(x)$
interlaces $p_1(x) + 2p_2(x) + \cdots + mp_m(x)$.  
\end{itemize}
\end{lemma}

\begin{proof} 
We only need to prove part (c) and for that, we proceed 
by induction on $m$. The case $m=1$ being trivial, let 
us assume that the result holds for a positive integer
$m-1$, consider a sequence $(p_1(x), 
p_2(x),\dots,p_m(x))$ and positive real numbers $c_1 \le c_2 
\le \cdots \le c_m$ as in the statement of the lemma 
and set $s_m(x) := p_1(x) + p_2(x) + \cdots + p_m(x)$. 
Since the sequence $(p_1(x),\dots,p_{m-2}(x),p_{m-1} + 
p_m(x))$ is also interlacing \cite[Lemma~3.4]{Fi06}, the 
induction hypothesis implies that $s_m(x)$ interlaces 
$c_1 p_1(x) + \cdots + c_{m-2} p_{m-2}(x) + c_{m-1} 
(p_{m-1}(x) + p_m(x))$. Since $s_m(x)$ also interlaces 
$(c_m - c_{m-1})p_m(x)$ (because each of its summands 
does so), it must interlace the sum of these two 
polynomials. This completes the induction.

For the second statement, let $s_{m-1}(x) := p_1(x) + 
p_2(x) + \cdots + p_{m-1}(x)$. From the first statement 
we have that $s_{m-1}(x)$ interlaces $p_1(x) + 2p_2(x) 
+ \cdots + (m-1)p_{m-1}(x)$. Since $s_{m-1}(x)$ also 
interlaces $mp_m(x)$, it must interlace the sum of 
these two polynomials and the proof follows.
\end{proof}

Every polynomial $p(x) \in \RR[x]$ of degree at most
$n$ can be written uniquely in the form $p(x) = a(x) + 
xb(x)$, where $a(x)$ and $b(x)$ are symmetric with 
centers of symmetry $n/2$ and $(n-1)/2$, respectively. 
We say that $p(x)$ has a \emph{nonnegative symmetric 
decomposition} with respect to $n$, if $a(x)$ and 
$b(x)$ have nonnegative coefficients. Following 
\cite{BS20}, we also say that $p(x)$ has a
\emph{real-rooted symmetric decomposition} 
(respectively, \emph{real-rooted and interlacing 
symmetric decomposition}) with respect to $n$, if 
$a(x)$ and $b(x)$ are real-rooted (respectively, if 
$a(x)$ and $b(x)$ are real-rooted and $x^n p(1/x)$ 
interlaces $p(x)$). By \cite[Theorem~2.6]{BS20}, if
$p(x)$ has a nonnegative, real-rooted and interlacing 
symmetric decomposition with respect to $n$, then 
$b(x)$ interlaces $a(x)$ and each one of them interlaces 
$p(x)$. The alternatingly increasing property for $p(x)$, 
defined earlier, with respect to $n$ is equivalent to 
the unimodality of both $a(x)$ and $b(x)$.

\section{The antiprism construction}
\label{sec:aconstr}

The antiprism triangulation of a simplicial complex
can be defined geometrically by iterating the 
antiprism construction. This section reviews the 
latter and studies its face enumeration, in the 
framework of uniform triangulations \cite{Ath20+}. 
The results will be applied in 
Section~\ref{sec:enumerate}, but may be of 
independent interest too.

Let $V = \{v_1, v_2,\dots,v_n\}$ be an $n$-element 
set and $\Delta$ be a triangulation of the boundary
complex of the simplex $2^V$. We pick an $n$-element 
set $U = \{u_1, u_2,\dots,u_n\}$ which is disjoint 
from the vertex set of $\Delta$ and denote by 
$\Gamma_\aA(\Delta)$ the collection of faces of
$\Delta$ together with all sets of the form $E \cup 
G$, where $E = \{ u_i~:~ i \in I\}$ is a nonempty face 
of the simplex $2^U$ for some $\varnothing \subsetneq
I \subseteq [n]$ and $G$ is a face of the restriction 
of $\Delta$ to the face $F = \{ v_i~:~ i \in [n] \sm I\}$ 
of $\partial (2^V)$ which is complementary to $E$. 
The collection $\Gamma_\aA(\Delta)$ is a simplicial 
complex which contains $2^U$ and $\Delta$ as 
subcomplexes; we call it the \emph{antiprism over 
$\Delta$}. When $\Delta = \partial(2^V)$ is the 
trivial triangulation, the antiprism $\Gamma_\aA
(\partial(2^V))$ is combinatorially isomorphic to 
the Schlegel diagram \cite[Section~5.2]{Zie95} of 
the $n$-dimensional cross-polytope behind any of its
facets. For general $\Delta$, the antiprism 
$\Gamma_\aA(\Delta)$ is a triangulation of 
$\Gamma_\aA(\partial(2^V))$: the carrier of a face 
$E \cup G$, as above, is the union of $E$ with the 
carrier of $G$, the latter considered as a face of 
the triangulation $\Delta$ of $\partial(2^V)$. Since 
$\Gamma_\aA(\partial(2^V))$ triangulates the simplex 
$2^V$, $\Gamma_\aA(\Delta)$ is a triangulation of 
$2^V$ as well with boundary complex equal to $\Delta$. 

\begin{remark} \label{rem:DeltaA} \rm
Given a triangulation $\Gamma$ of the 
$(n-1)$-dimensional simplex $2^V$, an analogous 
procedure defines a triangulation, say $\Delta_\aA
(\Gamma)$, of the $(n-1)$-dimensional sphere which 
contains $2^U$ and $\Gamma$ as subcomplexes
and which we may call the \emph{antiprism over 
$\Gamma$}. This construction was employed in 
\cite[Section~4]{Ath12}, in order to relate the 
$\gamma$-vector of a flag triangulation of the 
sphere to the local $\gamma$-vector of a flag 
triangulation of the simplex, and 
in~\cite[Section~4]{Ath20}, in order to interpret 
geometrically binomial Eulerian polynomials (see 
Example~\ref{ex:binom-euler}) and 
certain analogues for $r$-colored permutations. 
The connection between the two constructions is 
that $\Delta_\aA(\Gamma) = \Gamma \cup \Gamma_\aA
(\partial \Gamma)$. 
\qed
\end{remark}

The following statement is closely related to 
\cite[Proposition~4.1]{Ath20}.
\begin{proposition} \label{prop:A(Delta)} 
The simplicial complex $\Gamma_\aA(\Delta)$ 
triangulates the $(n-1)$-dimensional simplex $2^V$ 
for every triangulation $\Delta$ of the boundary 
complex $\partial(2^V)$. Moreover,
\[ h(\Gamma_\aA(\Delta), x) \ = \ \sum_{F \subsetneq 
   V} x^{|F|} h(\Delta_F, 1/x). \]
\end{proposition}

\begin{proof} 
We have already commented on the first sentence.
For the second, using 
Proposition~\ref{prop:h-ballcirc} and the 
definition of the $h$-polynomial we find that
\begin{eqnarray*}
x^n h(\Gamma_\aA(\Delta), 1/x) & = & 
h^\circ(\Gamma_\aA(\Delta), x) \ = \ 
\sum_{G \in \Gamma_\aA(\Delta)^\circ} x^{|G|} 
(1-x)^{n-|G|} \\ & = & 
\sum_{\varnothing \ne E \subseteq U} 
\sum_{\scriptsize \begin{array}{c} G \in 
\Gamma_\aA(\Delta),
\\ G \cap U = E \end{array}} x^{|G|} (1-x)^{n-|G|}.
\end{eqnarray*}
By definition of $\Gamma_\aA(\Delta)$, the inner 
sum is equal to $x^{|E|} h(\Delta_F, x)$, where 
$F \subsetneq V$ is the face of $2^V$ which is 
complementary to $E$. Replacing $x$ by $1/x$ results 
in the proposed expression for $h(\Gamma_\aA(\Delta), 
x)$ and the proof follows.
\end{proof}

We now turn our attention to uniform triangulations
of $\partial (2^V)$.
\begin{proposition} \label{prop:A(uDelta)} 
For every $\fF$-uniform triangulation $\Delta$ of 
the boundary complex of an $(n-1)$-dimensional 
simplex $2^V$:

\begin{eqnarray}
h(\Gamma_\aA(\Delta), x) & = & \sum_{k=0}^{n-1} 
   {n \choose k} x^k h_\fF(\sigma_k, 1/x) 
	          \label{eq:h-A(Delta)a} \\
& = & \sum_{k=0}^{n-1} {n \choose k} 
	  \ell_\fF(\sigma_k, x) \left( (1+x)^{n-k} - x^{n-k} 
		\right) \label{eq:h-A(Delta)b}, \\
\ell_V(\Gamma_\aA(\Delta), x) & = & \sum_{k=0}^{n-1} 
   {n \choose k} \ell_\fF(\sigma_k, x) 
	 \left( (1+x)^{n-k} - 1 - x^{n-k} \right)		
	         \label{eq:ell-A(Delta)}, \\
h(\Gamma_\aA(\Delta), x) - h(\Delta, x) 
& = & \sum_{k=0}^{n-1} {n \choose k} \ell_\fF
  (\sigma_k, x) \left( (1+x)^{n-k} - 1 - x - 
	\cdots - x^{n-k} \right) \label{eq:h-hbd-A(Delta)}		
\\ & = & \sum_{k=0}^{n-1} 
   {n \choose k} h_\fF(\sigma_k, x) 
	 \left( x^{n-k} - x (x-1)^{n-k-1} \right)	.
	         \label{eq:h-hbd-h-A(Delta)}
\end{eqnarray}

\medskip
In particular, if all restrictions of $\Delta$ to 
proper faces of $2^V$ are regular triangulations,
then the polynomials $\ell_V(\Gamma_\aA(\Delta), x)$ and  
$h(\Gamma_\aA(\Delta), x) - h(\Delta, x)$ are 
unimodal and $h(\Gamma_\aA(\Delta), x)$ is alternatingly 
increasing with respect to $n-1$.
\end{proposition}

\begin{proof} 
Equation~(\ref{eq:h-A(Delta)a}) follows directly 
from Proposition~\ref{prop:A(Delta)}. To deduce 
Equation~(\ref{eq:h-A(Delta)b}) from that, we use 
(\ref{eq:h-localh}) to express $h_\fF(\sigma_k, 1/x)$ 
in terms of local $h$-polynomials, apply the symmetry 
property of the latter and change the order of
summation to obtain 
\begin{eqnarray*}
h(\Gamma_\aA(\Delta), x) & = & \sum_{k=0}^{n-1} 
   {n \choose k} x^k \sum_{j=0}^k {k \choose j} 
	 \ell_\fF(\sigma_j, 1/x) \ = \ 
	\sum_{k=0}^{n-1} {n \choose k} \sum_{j=0}^k
	x^{k-j} {k \choose j} \ell_\fF(\sigma_j, x) \\
& = & \sum_{j=0}^{n-1} \ell_\fF(\sigma_j, x) 
  \sum_{k=j}^{n-1} {n \choose k} {k \choose j} x^{k-j} 
\ = \ \sum_{j=0}^{n-1} {n \choose j} \ell_\fF(\sigma_j, x)
  \sum_{k=j}^{n-1} {n-j \choose n-k} x^{k-j} \\ & = & 
	\sum_{j=0}^{n-1} {n \choose j} 
	  \ell_\fF(\sigma_j, x) \left( (1+x)^{n-j} - x^{n-j} 
		\right).
\end{eqnarray*}
For the fourth and fifth step we have used the identity 
${n \choose k} {k \choose j} = {n \choose j} {n-j 
\choose n-k}$ and the binomial theorem, respectively. 

Alternatively,
Equation~(\ref{eq:h-A(Delta)b}) follows from an 
application of Stanley's locality formula 
\cite[Theorem~3.2]{Sta92} to $\Gamma_\aA(\Delta)$, 
considered as a triangulation of the antiprism 
$\Gamma_\aA(\partial (2^V))$ over the boundary 
complex of $2^V$. Equation~(\ref{eq:ell-A(Delta)}) 
follows when combining (\ref{eq:h-A(Delta)b}) with
\begin{equation} \label{eq:h-ell-sigma}
  h(\Gamma_\aA(\Delta), x) \ = \ 
	\ell_V(\Gamma_\aA(\Delta), x) 
	\, + \, \sum_{k=0}^{n-1} {n \choose k}
  \ell_\fF(\sigma_k, x), 
\end{equation}
the latter being (\ref{eq:h-localh}) applied to 
$\Gamma_\aA(\Delta)$. Equation~(\ref{eq:h-hbd-A(Delta)}) 
follows from~(\ref{eq:h-A(Delta)b}) and 
\begin{equation} \label{eq:h-bdsigma}
 h(\Delta, x) \ = \ \sum_{k=0}^{n-1} {n \choose k} 
   \ell_\fF(\sigma_k, x) 
   (1 + x + x^2 + \cdots + x^{n-k-1}), 
\end{equation}
which is also a consequence of \cite[Theorem~3.2]
{Sta92}; see \cite[Equation~(4.2)]{KMS19}. 
Equation~(\ref{eq:h-hbd-h-A(Delta)}) follows 
from~(\ref{eq:h-hbd-A(Delta)}) by expressing $\ell_\fF
(\sigma_k, x)$ in terms of the $h$-polynomials 
$h_\fF(\sigma_j, x)$, changing the order of 
summation and computing the inner sum, just as 
in the proof of Equation~(\ref{eq:h-A(Delta)b}); we
leave the details of this computation to the 
interested reader.

For the last statement we
note that, by the regularity assumption, $\ell_\fF
(\sigma_k, x)$ is (symmetric with center of symmetry 
$k/2$ and) unimodal for $0 \le k < n$. As a 
result, Equations~(\ref{eq:ell-A(Delta)}) 
and~(\ref{eq:h-hbd-A(Delta)}) imply the unimodality 
of $\ell_V(\Gamma_\aA(\Delta), x)$ and 
$h(\Gamma_\aA(\Delta), x) - h(\Delta, x)$, 
respectively, and Equations~(\ref{eq:h-hbd-A(Delta)}) 
and~(\ref{eq:h-bdsigma}) imply that the symmetric 
decomposition 
\[ h(\Gamma_\aA(\Delta), x) \ = \ h(\Delta, x) + 
   (h(\Gamma_\aA(\Delta), x) - h(\Delta, x)) \]
of $h(\Gamma_\aA(\Delta), x)$ with respect to $n-1$ 
is nonnegative and unimodal. The latter statement is 
	equivalent to $h(\Gamma_\aA(\Delta),x)$ being alternatingly
	increasing.
\end{proof}

\begin{remark} \label{rem:c(Delta)} \rm
Let $\Delta$ be as in 
Proposition~\ref{prop:A(uDelta)}. Since coning 
a simplicial complex does not affect the 
$h$-polynomial, the right-hand side 
of~(\ref{eq:h-bdsigma}) is also an expression for 
$h(u \ast\Delta, x)$, where $u \ast \Delta$ denotes 
the cone of $\Delta$ with apex $u$. The formula
\begin{equation} \label{eq:h-bdsigma-new}
  h(u \ast\Delta, x) \ = \ \sum_{k=0}^{n-1} 
  {n \choose k} h_\fF(\sigma_k, x) (x-1)^{n-k-1} 
\end{equation}
can be derived from that by expressing $\ell_\fF
(\sigma_k, x)$ in terms of the $h$-polynomials 
$h_\fF(\sigma_j, x)$, changing the order of 
summation and computing the inner sum, just as 
in the proof of Equations~\eqref{eq:h-A(Delta)b} 
and \eqref{eq:h-hbd-h-A(Delta)}
or, alternatively, by adapting the argument in 
the proof of Proposition~\ref{prop:A(Delta)}.
When $\Delta$ is the barycentric subdivision of
$\partial \sigma_n$, this yields the recursion 
\[ A_n(x) \ = \ \sum_{k=0}^{n-1} {n \choose k} 
   A_k(x) (x-1)^{n-k-1} \]
for the Eulerian polynomial $A_n(x)$, valid for 
$n \ge 1$. This appears as Equation~(2.7) in \cite{Fo08}.
\end{remark}

\begin{example} \label{ex:binom-euler} \rm
Suppose again that $\Delta$ is the barycentric 
subdivision of $\partial \sigma_n$. Then, 
Equation~(\ref{eq:h-A(Delta)a}) yields that
\[ h(\Gamma_\aA(\Delta), x) \ = \ \sum_{k=0}^{n-1} 
   {n \choose k} x^k A_k(1/x) \ = \ 1 \, + \, 
	 x \sum_{k=1}^{n-1} {n \choose k} A_k(x) 
	 \ = \ \widetilde{A}_n(x) - xA_n(x) \]
and $h(\Gamma_\aA(\Delta), x) - h_\fF(\partial \sigma_n)
= \widetilde{A}_n(x) - (1+x)A_n(x)$, where 
\[ \widetilde{A}_n(x) \ := \ 1 \, + \, x 
   \sum_{k=1}^n {n \choose k} A_k(x) \]
is the $n$th binomial Eulerian polynomial studied, 
for instance, in \cite{Ath20, SW20}. From 
Equation~(\ref{eq:h-ell-sigma}) we compute 
further that $\ell_V(\Gamma_\aA(\Delta), x) = 
\widetilde{A}_n(x) - (1+x)A_n(x) - d_n(x)$, 
where $d_n(x) = \ell_\fF(\sigma_n, x)$ is the 
$n$th derangement polynomial (see 
Section~\ref{sec:triang}). 

Therefore, by Proposition~\ref{prop:A(uDelta)}, 
$\widetilde{A}_n(x) - xA_n(x)$ is alternatingly 
increasing with respect to $n-1$ and 
$\widetilde{A}_n(x) - (1+x)A_n(x)$ is symmetric 
and unimodal. 
\qed
\end{example}

\section{The antiprism triangulation}
\label{sec:atriang}

This section briefly describes combinatorially 
and geometrically the antiprism triangulation of a 
simplicial complex. For more information we refer
to \cite[Apendix~A.1]{IJ03} and 
\cite{Ko12}, where these descriptions are given in 
variant forms. We first review the corresponding 
descriptions of the barycentric subdivision, which we 
will parallel to treat the antiprism triangulation.

Let $\Delta$ be a simplicial 
complex. Consider the (simple, undirected) graph 
$\gG(\Delta)$ on the node set of nonempty faces 
of $\Delta$ for which two nodes are adjacent if one
is contained in the other. The barycentric subdivision 
$\sd(\Delta)$ is defined as the \emph{clique complex}
of $\gG(\Delta)$, meaning the abstract simplicial 
complex whose vertices are the nodes of $\gG(\Delta)$
and whose faces are the sets consisting of pairwise 
adjacent nodes. This is equivalent to the 
definition already given in Section~\ref{sec:triang}.

Geometrically, $\sd(\Delta)$
can be described as a triangulation of $\Delta$ as 
follows. Assume that all faces of $\Delta$ of 
dimension at most $j$ have been triangulated, for 
some $j \in \NN$. Then, triangulate each 
$(j+1)$-dimensional face of $\Delta$ by inserting 
one point in the interior of that face and coning 
over its boundary, which is already triangulated. 
By repeating this process, starting at $j=0$ and   
moving to higher dimensional faces, we get a 
triangulation of $\Delta$ which is 
combinatorially isomorphic to $\sd(\Delta)$. 
Alternatively, $\sd(\Delta)$ can be constructed
by applying successively the operation of stellar 
subdivision to each face of $\Delta$ of positive 
dimension, starting from the facets and moving to 
lower dimensional faces in any order which respects
reverse inclusion. A \emph{stellar subdivision} on 
a face $F \in \Delta$ replaces $\Star_\Delta(F)$ by 
the join of $\link_\Delta(F)$ with the cone over 
$\partial(2^F)$. 

The antiprism triangulation can be defined 
similarly, if the nonempty faces of $\Delta$ are 
replaced by pointed faces and coning is replaced by 
the antiprism construction of 
Section~\ref{sec:aconstr}. Recall that a 
\emph{pointed subset} of a set $V$ is any pair 
$(S, v)$ such that $v \in S \subseteq V$. Similarly, 
a \emph{pointed face} of a simplicial complex $\Delta$ 
is any pair $(F,v)$ such that $F \in \Delta$ is a 
face and $v \in F$ is a chosen vertex. 
\begin{definition} \label{def:sdA-complex} 
Let $\Delta$ be a simplicial 
complex. We denote by $\gG_\aA(\Delta)$ the (simple,
undirected) graph on the node set of pointed faces 
of $\Delta$ for which two distinct pointed faces 
$(F,v)$ and $(F',v')$ are adjacent if 

\begin{itemize}
\itemsep=0pt
\item[$\bullet$]
$F = F'$, or

\item[$\bullet$]
$F \subsetneq F'$ and $v' \in (F' \sm F)$, or 

\item[$\bullet$]
$F' \subsetneq F$ and $v \in (F \sm F')$.
\end{itemize}
The \emph{antiprism triangulation} of $\Delta$, denoted 
by $\sd_\aA(\Delta)$, is the abstract simplicial 
complex defined as the clique complex of $\gG_\aA
(\Delta)$. 
\end{definition}
Examples of antiprism triangulations are shown in 
Figures~\ref{fig:2simplex} and~\ref{fig:Bsp2}. 

\begin{figure}
\centering
  \includegraphics[width=0.7\textwidth]{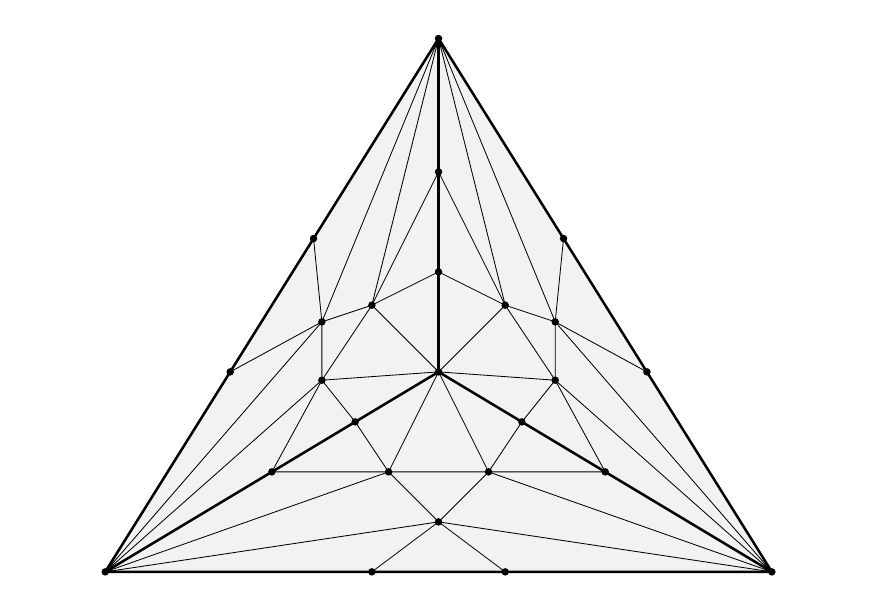}
  \caption{Antiprism triangulation of the cone 
	over the boundary of the 2-simplex}
  \label{fig:Bsp2}
\end{figure}

The faces of $\sd_\aA(\Delta)$ can be described 
explicitly, in combinatorial terms 
\cite[Section~2]{Ko12}. Given a set $S$, an 
\emph{ordered set partition} (or simply, ordered 
partition) of $S$ is any sequence of nonempty, 
pairwise disjoint sets (called \emph{blocks}) whose 
union is equal to $S$. A \emph{multi-pointed 
ordered partition} of $S$ is defined as a pair 
$(\pi, \tau)$, where $\pi = (B_1, B_2,\dots,B_m)$ 
and $\tau = (C_1, C_2,\dots,C_m)$ are ordered 
partitions of $S$ and of a subset of $S$, 
respectively, with the same number of blocks, 
such that $C_i$ is a nonempty subset of $B_i$ for 
every $i \in [m]$. We think of such a pair as an
ordered partition of $S$, together with a choice 
of a nonempty subset for every block. The sum of the 
cardinalities of these subsets $C_i$ (total number
of chosen elements) will be called the \emph{weight}
of $(\pi, \tau)$. Then, the $(k-1)$-dimensional 
faces of $\sd_\aA(\Delta)$ are in one-to-one 
correspondence with the multi-pointed ordered 
partitions of faces of $\Delta$ of weight $k$. 
More specifically, the multi-pointed ordered 
partition $(\pi, \tau)$, with $\pi = (B_1, 
B_2,\dots,B_m)$ and $\tau = (C_1, C_2,\dots,C_m)$,
corresponds to the face of $\sd_\aA(\Delta)$ with
vertices the pointed faces $(F, v)$ of $\Delta$, 
where $F = B_1 \cup B_2 \cup \cdots \cup B_i$ for
some $i \in [m]$ and $v \in C_i$. 
The faces of the antiprism triangulation of the simplex 
$2^V$ are the multi-pointed ordered partitions of
subsets of $V$; they will be referred to as 
\emph{multi-pointed partial ordered partitions} of 
$V$. Note that the facets of $\sd_\aA(\Delta)$ are 
in one-to-one correspondence with the ordered 
partitions of the facets of $\Delta$ (since all 
elements in the blocks should be chosen).
Figure \ref{fig:2simplex} shows the antiprism 
triangulation of the 2-simplex, including some faces 
labeled by multi-pointed ordered partitions.

As was the case with barycentric subdivision, 
$\sd_\aA(\Delta)$ can be constructed geometrically 
by applying the antiprism construction of 
Section~\ref{sec:aconstr} to its faces, starting 
from the edges and moving to faces of higher 
dimension in any order which respects inclusion. 
This process is slightly different from the one 
in \cite{IJ03, Ko12} which uses crossing operations 
on the faces of $\Delta$ instead, starting from 
facets and moving to faces of lower dimension in 
any order which respects reverse inclusion. A 
\emph{crossing operation} (also known as a 
\emph{balanced stellar subdivision} \cite{BM87}) 
on a face $F \in \Delta$ replaces $\Star_\Delta(F)$ 
by the join of $\link_\Delta(F)$ with the antiprism 
(as defined in Section~\ref{sec:aconstr}) over 
$\partial(2^F)$. Both approaches result in a 
triangulation which is combinatorially isomorphic 
to $\sd_\aA(\Delta)$. Under this isomorphism, 
the carrier of a multi-pointed ordered partition 
of a face $F \in \Delta$ is equal to $F$. As a 
result, the interior faces of the antiprism 
triangulation of the simplex $2^F$ are in 
one-to-one correspondence with the multi-pointed
ordered partitions of $F$. A type of operation 
more general than stellar and balanced stellar 
subdivision was introduced in~\cite{HN16} and
was applied there to all faces of a fixed 
dimension to produce a triangulation of $\Delta$.

\section{Face enumeration} 
\label{sec:enumerate}

This section studies the rich enumerative 
combinatorics of antiprism triangulations and 
proves Theorem~\ref{thm:mainA}. Following the 
notation of \cite{Ath20+}, we denote by 
$h_\aA(\sigma_n, x)$ and $\ell_\aA(\sigma_n, x)$
the $h$-polynomial and local $h$-polynomial of
$\sd_\aA(\sigma_n)$, respectively. These two
polynomials play an important role in this study.
The main difficulty for proving the real-rootedness
of $h_\aA(\sigma_n, x)$ comes from the fact that 
we know of no simpler recurrence relation for it 
than that of \Cref{prop:h-sdA(simplex)}. Some 
of the combinatorial interpretations of $h_\aA
(\sigma_n,x)$ extend to describe the effect of 
the antiprism triangulation on the $h$-polynomial 
of any simplicial complex.

\subsection{The antiprism triangulation of a 
simplex} 
\label{sec:enu-simplex}

As discussed in Section~\ref{sec:atriang}, the 
number of $(k-1)$-dimensional faces of the 
antiprism triangulation $\sd_\aA(\sigma_n)$ is 
equal to the number of multi-pointed partial 
ordered set partitions of $[n]$ of weight $k$. 
We now give a recurrence and combinatorial 
interpretations for the $h$-polynomial of 
$\sd_\aA(\sigma_n)$. For the first few values 
of $n$, 
 
\[  h_\aA(\sigma_n, x) \ = \ \begin{cases}
    1, & \text{if $n = 0$} \\
		1, & \text{if $n = 1$} \\
    1 + 2x, & \text{if $n=2$} \\
    1 + 9x + 3x^2, & \text{if $n=3$} \\
    1 + 28x + 42x^2 + 4x^3, & \text{if $n=4$} \\
    1 + 75x + 310x^2 + 150x^3 + 5x^4, & 
                            \text{if $n=5$} \\
    1 + 186x + 1725x^2 + 2300x^3 + 465x^4 + 6x^5, 
                           & \text{if $n=6$} \\
		1 + 441x + 8211x^2 + 23625x^3 + 13685x^4 + 
		    1323x^5 + 7x^6, & \text{if $n=7$.}
    \end{cases}  \]
		
We first need to introduce some more terminology.
Let $\varphi = (\pi, \tau)$ be a multi-pointed 
partial ordered set partition of $[n]$. Thus, $\pi 
= (B_1, B_2,\dots,B_m)$ is an ordered partition of 
a subset $S$ of $[n]$ and $\tau = (C_1, 
C_2,\dots,C_m)$, where $C_i$ is a nonempty subset
of $B_i$ for every $i \in [m]$. We will say that
$\varphi$ is \emph{proper} if $C_i$ is a proper 
subset of $B_i$ for every $i \in [m]$. We will 
use the same terminology with the adjective 
`partial' dropped, when $S = [n]$. The 
\emph{excedance set} of a permutation $w \in \fS_n$
is defined as the set of indices $i \in [n-1]$ such
that $w(i) > i$; see \cite{ES00} for more 
information on this concept.
\begin{proposition} \label{prop:h-sdA(simplex)} 
\begin{itemize}
\itemsep=0pt
\item[{\rm (a)}]
We have 
\begin{equation} \label{eq:hA(simplex)-rec}
h_\aA(\sigma_n, x) \ = \ \sum_{k=0}^{n-1} 
{n \choose k} x^k h_\aA(\sigma_k, 1/x)
\end{equation}
for every positive integer $n$.

\item[{\rm (b)}]
The coefficient of $x^k$ in $h_\aA(\sigma_n, x)$ 
is equal to:
\begin{itemize}
\itemsep=0pt
\item[$\bullet$]
the number of proper multi-pointed partial ordered 
set partitions of $[n]$ of weight $k$,

\item[$\bullet$]
the number of ways to choose a subset $S \subseteq 
[n]$ and an ordered set partition $\pi$ of $S$ and 
to color $k$ elements of $S$ black and the remaining 
elements white, so that no block of $\pi$ is 
monochromatic,

\item[$\bullet$]
the number of ordered set partitions $\pi = 
(B_1, B_2,\dots,B_m)$ of $[n]$ for which the 
union $\bigcup_{i=1}^{\lfloor m/2 \rfloor} B_i$
has exactly $k$ elements,

\item[$\bullet$]
${n \choose k}$ times the number of permutations
in $\fS_n$ with excedance set equal to $[k]$, 

\item[$\bullet$]
the explicit expression $$ {n \choose k} 
\sum_{j=1}^{k+1} \, (-1)^{k+1-j} j! S(k+1,j) 
j^{n-k-1}, $$
\end{itemize}
where $S(n,k)$ are the Stirling numbers of the 
second kind.
\end{itemize}
\end{proposition}

\begin{proof} 
Part (a) follows from 
Proposition~\ref{prop:A(uDelta)}, as a special case 
of Equation~(\ref{eq:h-A(Delta)a}). For part (b), 
we first note that from Equation \eqref{eq:ell-A(Delta)} 
of the same proposition and 
Equation~(\ref{eq:h-localh}) we get 
\[ \ell_\aA(\sigma_n, x) \ = \ \sum_{m=0}^{n-1} 
   {n \choose m} \ell_\aA(\sigma_m, x) 
	 \left( (1+x)^{n-m} - 1 - x^{n-m} \right)	\]
for $n \ge 1$ and 
\[ h_\aA(\sigma_n, x) \ = \ \sum_{m=0}^n
   {n \choose m} \ell_\aA(\sigma_m, x),	\]
respectively. By induction on $n$, the former 
equality implies that the coefficient of $x^k$ in 
$\ell_\aA(\sigma_n, x)$ is equal to the number of 
proper multi-pointed ordered set partitions of 
$[n]$ of weight $k$. This and the latter equation
yield the first interpretation of $h_\aA(\sigma_n, 
x)$ claimed in part (b). The second interpretation
is a restatement of the first (where black elements
correspond to the chosen elements in the blocks of
the multi-pointed partition).

The third interpretation can be deduced from the 
first as follows. Let $Q(n,k)$ denote the collection 
of proper multi-pointed partial ordered partitions 
of $[n]$ of weight $k$. Each element of $Q(n,k)$ 
is a triple consisting of a subset $S \subseteq [n]$, 
an ordered partition $\pi = (B_1, B_2,\dots,B_r)$ 
of $S$ and a choice of nonempty proper subset $C_i$ 
of $B_i$ for every $i \in [r]$,
such that the union $\cup_{i=1}^r C_i$ has 
cardinality $k$. From such a triple one can define 
an ordered partition of $[n]$ by listing the blocks 
$C_1,\dots,C_r, B_1 \sm C_1,\dots,B_r \sm C_r$ in 
this order and, if nonempty, adding $[n] \sm S$
at the end as the last block. It is straightforward 
to verify that the resulting map is a bijection 
from $Q(n,k)$ to the collection of ordered partitions 
of $[n]$ described in the third proposed 
interpretation.

For the last two claimed interpretations, let us 
denote by $c(n,k)$ the number of permutations in 
$\fS_n$ with excedance set equal to $[k]$, 
for $k \in \{0, 1,\dots,n\}$. Then, $c(n,n) = 0$
and, as a consequence of Lemma~2.2 and Theorem~2.5 
in \cite{ES00} (see also Section~3 of this 
reference), $c(n,k) = c(n,n-k-1)$ and 
\[ c(n, k) \ = \ 1 \, + \sum_{m=1}^k 
   {k+1 \choose m} c(n-k-1+m, m) \]
for $k \in \{0, 1,\dots,n-1\}$. In view of $c(n,k) 
= c(n,n-k-1)$, the latter equality can be 
rewritten as 
\begin{equation} \label{eq:c(n,k)-rec}
c(n, k) \ = \ 1 \, + \sum_{m=1}^{n-k-1} {n-k 
\choose m} c(k+m, m).
\end{equation}
On the other hand, writing $h_\aA(\sigma_n, x) = 
\sum_{k=0}^n p_\aA(n,k) x^k$ for $n \in \NN$, the 
recursion of part (a) gives that
\[ p_\aA(n, k) \ = \ \sum_{m=k}^{n-1} {n \choose m} 
   p_\aA(m, m-k) \]
for $k \in \{0, 1,\dots,n-1\}$. Setting $$ p_\aA(n,k) 
\ = \ {n \choose k} \bar{p}_\aA(n,k), $$ the last 
recursion can be rewritten as 
\begin{eqnarray*} 
{n \choose k} \bar{p}_\aA(n,k) & = & 
\sum_{m=k}^{n-1} {n \choose m} {m \choose k} 
\bar{p}_\aA(m, m-k) \ \ \text{i.e.,}\\
\bar{p}_\aA(n,k) & = & \sum_{m=k}^{n-1} 
{n-k \choose m-k} \bar{p}_\aA(m, m-k) \ = \
\sum_{m=0}^{n-k-1} {n-k \choose m} \bar{p}_\aA(k+m, m).
\end{eqnarray*}
Comparing this recursion to~(\ref{eq:c(n,k)-rec})
we get that $\bar{p}_\aA(n, k) = c(n,k)$ for all $n$ 
and all $0\leq k\leq n$.
This proves the next to last interpretation, claimed
in part (b). The last interpretation follows from 
this and the explicit formula for $c(n, k)$ obtained
in \cite[Proposition~6.5]{ES00}.
\end{proof}

The following statement is the main result of this 
section.
\begin{theorem} \label{thm:h-sdA(simplex)-roots}
The polynomial $h_\aA(\sigma_n, x)$ is real-rooted  
and interlaces $h_\aA(\sigma_{n+1}, x)$ for every 
$n \in \NN$. Moreover, $h_\aA(\sigma_n, x)$ has a
nonnegative, real-rooted and interlacing symmetric 
decomposition with respect to $n-1$, for every 
positive integer $n$.
\end{theorem}

\begin{proof} 
We consider the polynomials 
\[ q_{n,r} (x) \ := \ \sum_{k=0}^n {n \choose k} 
   x^{k+r} h_\aA(\sigma_{k+r}, 1/x), \]
shown in Table~\ref{tab:qnr} for small
values of $n, r \in \NN$. By part (a) of
Proposition~\ref{prop:h-sdA(simplex)} and the
definition of $q_{n,r}(x)$ we have 
\begin{eqnarray} 
q_{n,0} (x) & = & h_\aA(\sigma_n, x) + 
         x^n h_\aA(\sigma_n, 1/x), \label{eq:qn0} \\
q_{0,r} (x) & = & x^r h_\aA(\sigma_r, 1/x) 
                                   \label{eq:q0n} 
\end{eqnarray}
for every positive integer $n$ and every $r \in \NN$,
respectively. We claim that 
\[ \qQ_n \ := \ (q_{n,0}(x), 
   q_{n-1,1}(x),\dots,q_{1,n-1}(x), q_{0,n}(x), 
	 q_{0,n+1}(x)) \]
is an interlacing sequence of real-rooted polynomials
for every $n \in \NN$. In particular, selecting the 
first and last two terms, we have the interlacing 
sequence 
\[ (h_\aA(\sigma_n, x) + x^n h_\aA(\sigma_n, 1/x), 
    x^n h_\aA(\sigma_n, 1/x), x^{n+1} 
		h_\aA(\sigma_{n+1}, 1/x)) \]
of real-rooted polynomials for every $n \in \NN$. 
Before we prove the claim let us observe that, 
since $x^n h_\aA(\sigma_n, 1/x)$ and $x^{n+1} 
h_\aA(\sigma_{n+1}, 1/x)$ have degrees $n$ and 
$n+1$, respectively, the 
statement that the former polynomial interlaces
the latter is equivalent to the statement that 
$h_\aA(\sigma_n, x)$ interlaces $h_\aA(\sigma_{n+1}, 
x)$. Similarly, since $h_\aA(\sigma_n, x) + x^n 
h_\aA(\sigma_n, 1/x)$ is symmetric of degree $n$,
the statement that this polynomial interlaces
$x^{n+1} h_\aA(\sigma_{n+1}, 1/x)$ is equivalent 
to each of the statements that the same polynomial 
is interlaced by $x^n h_\aA(\sigma_{n+1}, 1/x)$ 
and that it interlaces $h_\aA(\sigma_{n+1}, x)$.

We now prove the claim by induction on $n$. This
is true for $n=0$, since $\qQ_0 = (1, x)$. We  
assume that it holds for $n-1 \in \NN$. The 
standard recurrence for the binomial coefficients
shows that $q_{n,r} (x) = q_{n-1,r} (x) + 
q_{n-1,r+1}(x)$ for every $r \in \NN$. Writing 
this in the form 
\[ q_{n-r,r} (x) \ = \ q_{n-r-1,r} (x) + 
                       q_{n-r-1,r+1} (x) \]
and iterating, we get
\begin{equation} \label{eq:q-rec}
q_{n-r,r} (x) \ = \ q_{n-r-1,r} (x) + q_{n-r-2,r+1} 
     (x) + \cdots + q_{0,n-1} (x) + q_{0,n} (x)
\end{equation}
for $r \in \{0, 1,\dots, n\}$. This means that the
first $n+1$ terms of $\qQ_n$ are the partial sums
of the reverse of $\qQ_{n-1}$ and hence they form 
an interlacing sequence, by part (b) of 
Lemma~\ref{lem:inter}. Thus, by part (a) of this 
lemma, to complete the induction it suffices 
to show that $q_{n,0} (x)$ and $q_{0,n} (x)$ 
interlace $q_{0,n+1} (x)$. As already discussed, 
and in view of~(\ref{eq:qn0}) and~(\ref{eq:q0n}),
this is equivalent to showing that 
$h_\aA(\sigma_n, x) + x^n h_\aA(\sigma_n, 1/x)$ 
and $h_\aA(\sigma_n, x)$ interlace 
$h_\aA(\sigma_{n+1}, x)$. To verify this we note 
that, setting $r=0$ in Equation~(\ref{eq:q-rec}), 
comparing with~(\ref{eq:qn0}) and~(\ref{eq:q0n}) 
and replacing $n$ with $n+1$, we get 
\begin{equation} 
q_{n,0} (x) + q_{n-1,1} (x) + \cdots + q_{0,n} (x)
\ = \ h_\aA(\sigma_{n+1}, x). \label{eq:qsum} 
\end{equation}
Since the sum of the terms of an interlacing 
sequence is interlaced by the first term, we 
conclude that $h_\aA(\sigma_n, x) + x^n 
h_\aA(\sigma_n, 1/x)$ interlaces 
$h_\aA(\sigma_{n+1}, x)$. Finally, applying
part (c) of Lemma~\ref{lem:inter} to the 
interlacing sequence $\qQ_{n-1}$ we conclude
that the sum of the first $n$ terms of this 
sequence, which equals $h_\aA(\sigma_n, x)$, 
interlaces the sum of the partial sums of the 
reverse of $\qQ_{n-1}$, which equals 
$h_\aA(\sigma_{n+1}, x)$. This completes the 
proof of the claim.

Finally, note that $x^n h_\aA(\sigma_n, 1/x)$ 
and $x^{n+1} h_\aA(\sigma_{n+1}, 1/x)$ are the 
last two terms of $\qQ_n$. Since this sequence 
is interlacing, the two polynomials are 
real-rooted and the former interlaces the 
latter. As already discussed, this means that 
$h_\aA(\sigma_n, x)$ is real-rooted and 
interlaces $h_\aA(\sigma_{n+1}, x)$. Similarly,
the sum of the first $n$ terms of the sequence
$\qQ_{n-1}$ interlaces the last term. In view 
of~(\ref{eq:q0n}) and~(\ref{eq:qsum}), this 
means that $h_\aA(\sigma_n, x)$ interlaces 
$x^n h_\aA(\sigma_n, 1/x)$ and, equivalently, 
that $h_\aA(\sigma_n, x)$ is interlaced by  
$x^{n-1} h_\aA(\sigma_n, 1/x)$. Since we already 
know from Proposition~\ref{prop:A(uDelta)} 
that $h_\aA(\sigma_n, x)$ has a nonnegative 
symmetric decomposition with respect to $n-1$, 
this decomposition must be real-rooted and 
interlacing by \cite[Theorem~2.6]{BS20}. 
\end{proof}

{\scriptsize
\begin{table}[hptb]
\begin{center}
\begin{tabular}{| l || l | l | l | l | l |} \hline
& $r=0$ & $r=1$ & $r=2$    & $r=3$ \\ \hline \hline
$n=0$   & 1    & $x$    & $2x+x^2$ & $3x+9x^2+x^3$ \\ 
         \hline
$n=1$   & $1+x$ & $3x+x^2$ & $5x+10x^2+x^3$ & 
          $7x+51x^2+29x^3+x^4$ \\ \hline
 $n=2$  & $1+4x+x^2$ & $8x+11x^2+x^3$ & 
          $12x+61x^2+30x^3+x^4$ & \\ \hline
 $n=3$  & $1+12x+12x^2+x^3$ & $20x+72x^2+31x^3+x^4$ 
        & & \\ \hline
\end{tabular}
\caption{Some polynomials $q_{n,r}(x)$.}
\label{tab:qnr}
\end{center}
\end{table}
}

Let us write $\theta_\aA(\sigma_n, x) := h_\aA
(\sigma_n, x) - h_\aA(\partial \sigma_n, x)$. 
As mentioned in the proof of 
Proposition~\ref{prop:A(uDelta)}, the expression
$h_\aA(\sigma_n, x) = h_\aA(\partial \sigma_n, x) 
+ \theta_\aA(\sigma_n, x)$ is the (nonnegative) 
symmetric decomposition of $h_\aA(\sigma_n, x)$ 
with respect
to $n-1$. Thus, $h_\aA(\partial \sigma_n, x)$ and 
$\theta_\aA(\sigma_n, x)$ are real-rooted by 
Theorem~\ref{thm:h-sdA(simplex)-roots}. Although 
the latter appears to be a very special case of
Conjecture~\ref{conj:main}, according to 
\cite[Theorem~1.2]{Ath20+}, it would imply the  
conjecture if the following statement (which we 
have verified computationally for $n \le 20$) 
also turns out to be true.
\begin{conjecture} \label{conj:thetaA}
The polynomial $h_\aA(\sigma_{n-1}, x)$ interlaces
$\theta_\aA(\sigma_n, x)$ for every positive 
integer $n$. 
\end{conjecture}

\begin{remark} \label{rem} \rm
The polynomial $h_\aA(\sigma_n, x) + x^n h_\aA
(\sigma_n, 1/x)$, shown to be real-rooted in the 
proof of Theorem~\ref{thm:h-sdA(simplex)-roots},
is equal to the $h$-polynomial of a flag 
triangulation of the $(n-1)$-dimensional sphere. 
Indeed, let $\Gamma = \sd_\aA(\sigma_n)$, so that 
$h(\Gamma, x) = h_\aA(\sigma_n, x)$. Then, 
in the notation of Section~\ref{sec:aconstr}, in 
particular Remark \ref{rem:DeltaA}, 
$\Delta = \Delta_\aA(\Gamma)$ is a flag 
triangulation of the $(n-1)$-dimensional sphere
and $h(\Delta, x) = h(\Gamma, x) + h^\circ(\Gamma, x) 
= h(\Gamma, x) + x^n h(\Gamma, 1/x) = 
h_\aA(\sigma_n, x) + x^n h_\aA(\sigma_n, 1/x)$.
\end{remark}

\begin{remark} \label{rem:bar{p}(x)} \rm
The polynomial 
\[ \bar{p}_\aA(\sigma_n, x) \ := \ \sum_{k=0}^n 
   \bar{p}_\aA(n,k) x^k \ = \ \sum_{k=0}^n c(n,k)  
	 x^k, \]
where, as in the proof of \Cref{prop:h-sdA(simplex)}, 
$c(n,k)$ is the number of permutations in 
$\fS_n$ with excedance set equal to $[k]$, was 
shown to be symmetric and unimodal in 
\cite[Section~3]{ES00}. For the first few values 
of $n$,
\[  \bar{p}_\aA(\sigma_n, x) \ = \ \begin{cases}
		1, & \text{if $n = 1$} \\
    1 + x, & \text{if $n=2$} \\
    1 + 3x + x^2, & \text{if $n=3$} \\
    1 + 7x + 7x^2 + x^3, & \text{if $n=4$} \\
    1 + 15x + 31x^2 + 15x^3 + x^4, & 
                            \text{if $n=5$} \\
    1 + 31x + 115x^2 + 115x^3 + 31x^4 + x^5, 
                           & \text{if $n=6$} \\
		1 + 63x + 391x^2 + 675x^3 + 391x^4 + 63x^5 + 
		            x^6, & \text{if $n=7$.}
    \end{cases}  \]	
	
The following statement is stronger than the 
real-rootedness of $h_\aA(\sigma_n, x)$.
\begin{conjecture} \label{conj:bar{p}(x)-roots}
The polynomial $\bar{p}_\aA(\sigma_n, x)$ is 
real-rooted and interlaces $\bar{p}_\aA(\sigma_{n+1},
x)$ for every $n \in \NN$. In particular, 
$\bar{p}_\aA(\sigma_n, x)$ is $\gamma$-positive
for every $n \in \NN$.
\end{conjecture}
\end{remark}

\subsection{The local $h$-polynomial} 
\label{sec:local}

We now focus on the local $h$-polynomial 
$\ell_\aA(\sigma_n, x)$ of the antiprism 
triangulation of $\sigma_n$. For the 
first few values of $n$,

\[  \ell_\aA(\sigma_n, x) \ = \ \begin{cases}
    1, & \text{if $n = 0$} \\
		0, & \text{if $n = 1$} \\
    2x, & \text{if $n=2$} \\
    3x + 3x^2, & \text{if $n=3$} \\
    4x + 30x^2 + 4x^3, & \text{if $n=4$} \\
    5x + 130x^2 + 130x^3 + 5x^4, & 
                            \text{if $n=5$} \\
    6x + 435x^2 + 1460x^3 + 435x^4 + 6x^5, 
                           & \text{if $n=6$} \\
		7x + 1281x^2 + 10535x^3 + 10535x^4 + 1281x^5
		   + 7x^6,            & \text{if $n=7$.}
    \end{cases}  \]

We now provide a recurrence, combinatorial 
interpretations and formulas for the polynomials 
$\ell_\aA(\sigma_n, x)$.
\begin{proposition} \label{prop:ell-sdA} 
\begin{itemize}
\itemsep=0pt
\item[{\rm (a)}]
We have 
\begin{equation} \label{eq:ellA-rec}
\ell_\aA(\sigma_n, x) \ = \ \sum_{k=0}^{n-1} 
{n \choose k} \ell_\aA(\sigma_k, x) 
	 \left( (1+x)^{n-k} - 1 - x^{n-k} \right)
\end{equation}
for every positive integer $n$. In particular, 
$\ell_\aA(\sigma_n, x)$ is unimodal for every $n 
\in \NN$.

\item[{\rm (b)}]
The coefficient of $x^k$ in $\ell_\aA(\sigma_n, x)$ 
is equal to:
\begin{itemize}
\itemsep=0pt
\item[$\bullet$]
the number of proper multi-pointed ordered set 
partitions of $[n]$ of weight $k$,

\item[$\bullet$]
the number of ways to choose an ordered set 
partition $\pi$ of $[n]$ and to color $k$ elements 
of $[n]$ black and the remaining $n-k$ white, so 
that no block of $\pi$ is monochromatic,

\item[$\bullet$]
the number of ordered set partitions $(B_1, B_2,\ldots,B_m)$
of $[n]$ having an even number of blocks for which the union
$\bigcup_{i=1}^{m/2}B_i$ has exactly $k$ elements,

\item[$\bullet$]
${n \choose k}$ times the number of derangements
in $\fS_n$ with excedance set equal to $[k]$, 

\item[$\bullet$]
the explicit expression $$ {n \choose k} 
\sum_{j \ge 1} (j!)^2 S(k,j) S(n-k,j), $$
\end{itemize}
where $S(n,k)$ are the Stirling numbers of the 
second kind.
\end{itemize}
\end{proposition}

\begin{proof} 
The recurrence of part (a) follows from 
Proposition~\ref{prop:A(uDelta)}, as a special case 
of Equation~(\ref{eq:ell-A(Delta)}). The unimodality
of $\ell_\aA(\sigma_n, x)$ follows directly from the 
recurrence by induction on $n$ (and, alternatively, 
from the regularity of the antiprism triangulation 
of the simplex; see the proof of 
\Cref{prop:aASimplexToBoundary}).

For part (b), the first interpretation was already 
shown in the proof of 
Proposition~\ref{prop:h-sdA(simplex)} and the 
second is a restatement of the first. The third 
interpretation follows from the first and the proof 
of the corresponding result of Proposition 
\ref{prop:h-sdA(simplex)} by noting that in the 
provided bijection the set $[n] \sm S$ is always 
empty. Furthermore, the fifth interpretation follows 
from the second one since there are ${n \choose k}$ 
ways to choose the $k$ black elements of $[n]$ and 
for every such choice and every $j \ge 1$, there 
are $j! S(k,j) \cdot j! S(n-k,j)$ ways to choose 
an ordered partition of $[n]$ with $j$ blocks, 
none of which is monochromatic.

Finally, we deduce the fourth interpretation from 
the corresponding result of part (b) of 
Proposition~\ref{prop:h-sdA(simplex)}. Let us  
use the notation adopted in the proof of that 
proposition, write 
$\ell_\aA(\sigma_n, x) = \sum_{k=0}^n \ell_\aA(n,k) 
x^k$ for $n \in \NN$ and set $$ \ell_\aA(n,k) \ = \ 
{n \choose k} \bar{\ell}_\aA(n,k). $$
Then, by the second interpretation, considering 
the elements $1, 2,\dots,k$ colored black and
the other elements of $[n]$ colored white, 
$\bar{\ell}_\aA(n,k)$ is equal to the number of 
ordered set partitions of $[n]$ with no 
monochromatic block. By 
Proposition~\ref{prop:h-sdA(simplex)}, under the 
same coloring convention, $\bar{p}_\aA(n,k)$ is 
equal to the number of ways
to choose a set $[k] \subseteq S \subseteq [n]$ 
and an ordered set partition of $S$ with no 
monochromatic block. These interpretations imply
that 
\[ \bar{p}_\aA(n,k) \ = \ \sum_{i=0}^{n-k} {n-k 
   \choose i} \bar{\ell}_\aA(n-i,k) \]
for all $n, k$. Denoting by $d(n,k)$ the number of 
derangements in $\fS_n$ with excedance set equal to 
$[k]$, it should also be clear that 
\[ c(n,k) \ = \ \sum_{i=0}^{n-k} {n-k \choose i} 
   d(n-i,k) \]
for all $n, k$. By 
Proposition~\ref{prop:h-sdA(simplex)}, we have 
$\bar{p}_\aA(n,k) = c(n,k)$ for all $n, k$. Therefore, 
the two expressions for these numbers above and an 
easy induction show that $\bar{\ell}_\aA(n,k) = 
d(n,k)$ for all $n, k$ and the proof follows.
\end{proof}

Following notation introduced in the previous 
proof, we set $\bar{\ell}_\aA(\sigma_n, x) := 
\sum_{k=0}^n \bar{\ell}_\aA(n,k) x^k$. We note that,
since the polynomial $\ell_\aA(\sigma_n, x)$ is 
symmetric with center of symmetry $n/2$, so is 
$\bar{\ell}_\aA(\sigma_n, x)$.
\begin{conjecture} \label{conj:ellA(x)-roots}
\begin{itemize}
\itemsep=0pt
\item[{\rm (a)}]
The polynomial $\ell_\aA(\sigma_n, x)$ is real-rooted 
and interlaces $\ell_\aA(\sigma_{n+1}, x)$ for every 
$n \in \NN$. 

\item[{\rm (b)}]
The polynomial $\bar{\ell}_\aA(\sigma_n, x)$ is 
real-rooted and interlaces $\bar{\ell}_\aA(\sigma_{n+1}, 
x)$ for every $n \in \NN$. In particular, 
$\bar{\ell}_\aA(\sigma_n, x)$ is $\gamma$-positive
for every $n \in \NN$.
\end{itemize}
\end{conjecture}

The following statement 
confirms a general conjecture of \cite{Ath12}, 
claiming that all flag triangulations of simplices 
have $\gamma$-positive local $h$-polynomials, in 
the special case of antiprism triangulations and 
provides evidence in favor 
of~\Cref{conj:ellA(x)-roots} (a). The 
$\gamma$-positivity of $\ell_\aA(\sigma_n, x)$
follows from that of the polynomial 
$\theta_\aA(\sigma_n, x) := h_\aA(\sigma_n, x) - 
h_\aA(\partial \sigma_n, x)$, which appeared in 
Conjecture~\ref{conj:thetaA}, and 
\cite[Theorem~4.4]{KMS19}. 
\begin{proposition} \label{prop:ell-sdA-gamma} 
The polynomial $\ell_\aA(\sigma_n, x)$ is 
$\gamma$-positive for every $n \in \NN$.
\end{proposition}

\begin{proof}
Theorem~4.4 in ~\cite{KMS19}, applied to $\sd_\aA
(\sigma_n)$, asserts that 
\[ \ell_\aA(\sigma_n, x) \ = \ \sum_{k=0}^n 
   {n \choose k} \theta_\aA (\sigma_k, x) d_{n-k}(x)
	 , \]
where $d_n(x)$ is the $n$th derangement polynomial 
discussed in Section~\ref{sec:triang}. The polynomial 
$\theta_\aA (\sigma_k, x)$ is symmetric, with center 
of symmetry $k/2$, and, as discussed after the proof
of Theorem~\ref{thm:h-sdA(simplex)-roots}, it has 
nonnegative coefficients and only real roots. As a 
result, it is
$\gamma$-positive, with center of symmetry $k/2$. 
The same property is shared by the derangement 
polynomial $d_k(x)$; see \cite[Theorem~2.13]{Ath18}.
Therefore, the right-hand side of the previous 
formula for $\ell_\aA(\sigma_n, x)$ is a sum of 
$\gamma$-positive polynomials, each with center of 
symmetry $k/2 + (n-k)/2 = n/2$, and the proof follows.
\end{proof}

A combinatorial interpretation of $\theta_\aA
(\sigma_n, x)$ can be deduced from 
Proposition~\ref{prop:ell-sdA}.
\begin{corollary} \label{prop:theta-sdA} 
The coefficient of $x^k$ in $\theta_\aA(\sigma_n, x)$ 
equals the number of ways to choose an ordered set 
partition $\pi$ of $[n]$ and to color $k$ elements of 
$[n]$ black and the remaining $n-k$ white, so that no 
block of $\pi$ is monochromatic and there is a black 
element which is larger than a white element in the 
last block of $\pi$. 
\end{corollary}

\begin{proof}
As a consequence of \cite[Lemma~4.1]{KMS19} (and as 
already discussed in the proof of 
Proposition~\ref{prop:A(uDelta)}), we have 
\[ \theta_\aA(\sigma_n, x) \ = \ \ell_\aA(\sigma_n, 
   x) \, - \, \sum_{m=0}^{n-2} {n \choose m} 
   \ell_\aA(\sigma_m, x) 
	 (x + x^2 + \cdots + x^{n-m-1}) \]
for every positive integer $n$. By the second 
interpretation of $\ell_\aA(\sigma_n, x)$ provided 
by Proposition~\ref{prop:ell-sdA} (b), the coefficient
of $x^k$ in the sum on the right-hand side is equal
to the number of ways to choose an ordered set 
partition $\pi$ of $[n]$ and to color $k$ elements of 
$[n]$ black and the remaining $n-k$ white, so that no 
block of $\pi$ is monochromatic and every black 
element in the last block of $\pi$ is smaller than 
every white element of that block. Thus, the proposed 
interpretation of $\theta_\aA(\sigma_n, x)$ follows 
from the previous equation and 
Proposition~\ref{prop:ell-sdA}.
\end{proof}

\subsection{Face-vector transformations} 
\label{sec:transform}

The general results of \cite{Ath20+} on uniform 
triangulations of simplicial complexes imply that 
there exist nonnegative integers $q_\aA(n,k)$ and 
$p_\aA(n,k,j)$ for $n, k, j \in \NN$ with $k, j 
\le n$ such that
\begin{equation}
\label{eq:f-transform}
f_{j-1}(\sd_\aA(\Delta)) \ = \ \sum_{k=j}^n 
       q_\aA(k,j) f_{k-1}(\Delta) 
\end{equation} 
and
\begin{equation}
\label{eq:h-transform}
h_j(\sd_\aA(\Delta)) \ = \ \sum_{k=0}^n p_\aA(n,k,j) 
    h_k(\Delta) 
\end{equation} 
for every $(n-1)$-dimensional simplicial complex 
$\Delta$ and every $j \in \{0, 1,\dots,n\}$. 
The former equation is easy to explain; a 
simple counting argument (see the proof of 
\cite[Theorem~4.1]{Ath20+}) shows its validity
when $q_\aA(n,k)$ is defined as the number of 
$(k-1)$-dimensional faces in the interior of the
antiprism triangulation of $\sigma_n$. This yields 
the following statement.
\begin{proposition} \label{prop:f-transform} 
For all integers $n \ge 1$ and $k \in \{0, 
1,\dots,n\}$, $q_\aA(n,k)$ is equal to the number 
of multi-pointed ordered set partitions of $[n]$ of 
weight $k$. Moreover, we have the explicit formula 
\[ q_\aA(n,k) \ = \ {n \choose k} \sum_{j=0}^k j!  
        S(k,j) j^{n-k} \]
where, as usual, $S(k,j)$ is a Stirling number of 
the second kind.
\end{proposition}

\begin{proof} 
The proposed combinatorial interpretation follows
from our previous discussion and that in 
Section~\ref{sec:atriang}. To verify the formula,
we note that there are ${n \choose k} \cdot j! S(k,j)$ 
ways to choose a $k$-element subset $S$ of $[n]$ and 
an ordered partition of $S$ with $j$ blocks and, for 
each such choice there are $j^{n-k}$ ways to 
distribute the remaining $n-k$ elements of $[n]$ in 
the blocks so as to form a multi-pointed ordered set 
partition of $[n]$ with set of chosen elements equal
to $S$.
\end{proof}

Equation~(\ref{eq:h-transform}) and the nonnegativity 
of the coefficients $p_\aA(n,k,j)$ which appear there 
are less obvious. Various interpretations, an explicit
formula and a recurrence are given 
for these numbers in \cite{Ath20+} in the general 
framework of uniform triangulations. In particular, 
as shown in \cite[Corollary~5.6]{Ath20+} (and 
originally by the second and third author of this 
paper), the recurrence 
\begin{equation}
\label{eq:recurrence}
p_\aA(n,k,j) \ = \ p_\aA(n,k-1,j) + 
p_\aA(n-1,k-1,j-1) - p_\aA(n-1,k-1,j)
\end{equation} 
holds for all $k, j \in \{0, 1,\dots,n\}$ with 
$k \ge 1$. We will also keep in mind that $p_\aA
(n,0,j) = p_\aA(n,j)$ is the coefficient of $x^j$ 
in $h_\aA(\sigma_n, x)$. This observation is the
special case $\Delta = \sigma_n$ of 
Equation~(\ref{eq:h-transform}).

The following combinatorial interpretations of 
$p_\aA(n,k,j)$ generalize some of those given for
$p_\aA(n,k)$ in 
Proposition~\ref{prop:h-sdA(simplex)}.
\begin{proposition} \label{prop:h-transform} 
For all integers $n \ge 1$ and $k \in 
\{0, 1,\dots,n\}$, $p_\aA(n,k,j)$ is equal to:
\begin{itemize}
\itemsep=0pt
\item[$\bullet$]
the number of ways to choose a set $[k] \subseteq 
S \subseteq [n]$ and an ordered set partition $\pi$ 
of $S$ and to color $j$ elements of $S$ black and 
the remaining elements white, so that the following 
condition holds: if a block $B$ of $\pi$ is 
monochromatic, then 
\begin{itemize}
\itemsep=0pt
\item[$\circ$] $B$ is the first block of $\pi$,
\item[$\circ$] $B \subseteq [k]$, and
\item[$\circ$] all elements of $B$ are colored black.
\end{itemize}

\item[$\bullet$]
the number of ordered set partitions $(B_1, 
B_2,\dots,B_m)$ of $[n]$ for which the following 
conditions hold: 
\begin{itemize}
\itemsep=0pt
\item[$\circ$] 
if $m$ is even, then $\bigcup_{i=1}^{\lfloor m/2 
\rfloor} B_i$ has exactly $j$ elements, and
\item[$\circ$] if $m$ is odd, then the union of
$\bigcup_{i=1}^{\lfloor m/2 \rfloor} B_i$ and $B_m 
\cap [k]$ has exactly $j$ elements.
\end{itemize}
\end{itemize}
\end{proposition}

\begin{proof} 
Let $Q(n,k,j)$ be the collection of triples of sets 
$S$, partitions of $S$ and colorings of the elements 
of $S$ described in the first proposed combinatorial 
interpretation of $p_\aA(n,k,j)$ and let $q(n,k,j)$ 
be the cardinality of $Q(n,k,j)$. 
We will show that $p_\aA(n,k,j) = q(n,k,j)$. This 
is true for $k=0$ by the first combinatorial 
interpretation of $p_\aA(n,j) = p_\aA(n,0,j)$ 
provided by Proposition~\ref{prop:h-sdA(simplex)}.
Thus, it suffices to show that the numbers $q(n,k,j)$ 
satisfy recurrence (\ref{eq:recurrence}) or, 
equivalently, that 
\[ q(n,k-1,j) \ = \ q(n,k,j) + q(n-1,k-1,j)  - 
   q(n-1,k-1,j-1) \]
for $k \ge 1$. By definition, $q(n,k-1,j)$ is the 
number of triples in the collection $Q(n,k-1,j)$, 
each one consisting of a set $S$, a partition of $S$ 
and a coloring of the elements of $S$ having certain 
properties. Clearly, we have $k \not\in S$ for 
exactly $q(n-1,k-1,j)$ of these triples. Moreover,
we have $k \in S$ for exactly $q(n,k,j) - 
q(n-1,k-1,j-1)$ of them, since for exactly 
$q(n-1,k-1,j-1)$ of the triples in $Q(n,k,j)$  
there is a monochromatic block which contains $k$.
This proves the first interpretation. 

As an alternative proof, by computing the 
coefficient of $x^j$ in the right-hand side of 
\cite[Equation~(12)]{Ath20+} we get the explicit 
expression
\[ p_\aA(n,k,j) \ = \ \sum_{r=0}^n \sum_{i=0}^j 
   {k \choose i} {n-k \choose n-r-i} 
	 \ell_\aA(r,j-i). \]
The double sum on the right side is also equal to
$q(n,k,j)$ since to choose a set $S$, a partition 
$\pi$ and a coloring as in the statement of the 
proposition so that a monochromatic block $B$ has 
exactly $i$ elements, 
if present, and there is a total of $r$ elements 
in the remaining blocks of $\pi$, there are 
${k \choose i}{n-k \choose n-r-i}$ ways to choose 
the $i$ elements of $B$ and the $n-r-i$ elements of
$[n]$ not in the blocks of $\pi$ and for each such 
choice, by Proposition~\ref{prop:ell-sdA}, there 
are $\ell_\aA(r,j-i)$ ways to choose the blocks of 
$\pi$ other than $B$.

To prove the second interpretation, it suffices to
find a bijection from $Q(n,k,j)$ to the collection 
of ordered set partitions described there. Such a 
bijection can be constructed as an obvious extension 
of the one provided in the proof of 
Proposition~\ref{prop:h-sdA(simplex)} for the 
special case $k=0$. More specifically, the elements 
of the monochromatic block, if present, of a colored 
ordered partition in $Q(n,k,j)$ should be included 
in the last block of the ordered partition produced 
by the bijection; the details are left to the 
interested reader.
\end{proof}

\section{Lefschetz properties}
\label{sec:Lefschetz}

This section reviews basic definitions and background 
on Lefschetz properties for simplicial complexes and 
includes some preliminary technical results, which 
will be applied in the following section in the context 
of antiprism triangulations.

Let $\Delta$ be an $(n-1)$-dimensional simplicial 
complex which is Cohen--Macaulay over an infinite 
field $\FF$ and let $s \le n$ be a positive 
integer. We say that $\Delta$ has the 
\emph{$s$-Lefschetz property} (over $\FF$) if 
there exists a linear system of parameters $\Theta$ 
for $\FF[\Delta]$ and a linear form $\omega \in 
\FF[\Delta]$, such that the multiplication maps
$$
\cdot\omega^{s-2i}:\left(\FF[\Delta]/
\Theta\FF[\Delta]\right)_i\to\left(\FF[\Delta]/
\Theta\FF[\Delta]\right)_{s-i}
$$
are injective for all $0 \le i \le \lfloor (s-1)/2 
\rfloor$. Following \cite{KN09}, we call $\Delta$ 
\emph{almost strong Lefschetz} (over $\FF$) if it
has the $(n-1)$-Lefschetz property. Usually, if 
$\Delta$ has the $n$-Lefschetz property 
and, additionally, the above multiplication maps 
are isomorphisms, one says that $\Delta$ is 
\emph{strong Lefschetz} (or $\Delta$ has the 
\emph{strong Lefschetz property}). Lefschetz 
properties are an important tool in the area of 
face enumeration of simplicial complexes; various 
classes of simplicial complexes, the most 
prominent probably being boundary complexes of 
simplicial polytopes \cite{Sta80}, are known to 
have such properties. Barycentric subdivisions 
of shellable simplicial complexes were shown 
in~\cite{KN09} to be almost strong Lefschetz 
over $\FF$. 

The proof of \Cref{thm:LefschetzUnimodal}, which 
follows similar lines, is fairly elementary and does 
not require heavy machinery. Since it is rather 
lengthy, we now explain the main steps to guide 
the reader through it. The main idea to show that 
the antiprism triangulation of a shellable simplicial
complex is almost strong Lefschetz is to use induction 
on the number of facets and the dimension. The 
inductive step (see \Cref{thm:LefschetzShellable}) 
essentially follows from a short exact sequence and 
some standard arguments for commutative diagrams. 
The hardest part is the base of the induction, 
namely to prove the almost strong Lefschetz
property for the antiprism triangulation of a 
simplex (see \Cref{thm:SimplexLefschetz}). The main
idea there is to show that $\sd_\aA(\sigma_n)$ is 
almost strong Lefschetz if and only if so is its 
boundary complex (\Cref{prop:aASimplexToBoundary}). 
Since the boundary complex can be realized as the 
boundary complex of a simplicial polytope 
(\Cref{prop:SLPSigma}), the claim follows from 
\cite{Sta80}. To prove \Cref{prop:aASimplexToBoundary}, 
we provide an explicit sequence of edge contractions 
which preserve the almost strong Lefschetz property  
and transform  $\partial\sd_\aA(\sigma_n)$ into 
$\sd_\aA(\sigma_n)$. 
It is known that Lefschetz properties behave well 
if the mentioned edge contractions are sufficiently 
nice. Before we can make this more precise, we need 
to introduce some definitions. 

Let $\Delta$ be a simplicial complex on a vertex 
set $V$ which is endowed with a total order $<$. 
Given an edge $e = \{a, b\}\in \Delta$ with $a < b$, 
the \emph{contraction} $\cC_\Delta(e)$ of $\Delta$ 
with respect to $e$ is the simplicial complex on 
the vertex set $V \sm \{b\}$ which is obtained 
from $\Delta$ by identifying vertices $a$ and $b$, 
i.e.,
\begin{equation*}
\cC_\Delta(e) \ := \ \{F \in \Delta~:~b \notin F\} 
\, \cup \, \{(F \sm \{b\})\cup\{a\}~:~b \in F \in 
\Delta\}.
\end{equation*}
We say that $\Delta$ satisfies the \emph{Link 
Condition} with respect to $e$ if 
\begin{equation*}
\link_\Delta(e) \ = \ \link_\Delta(\{a\}) \cap 
\link_\Delta(\{b\}).
\end{equation*}

\begin{proposition} \label{prop:LinkCondition}
Let $\FF$ be an infinite field and let $\Delta$ 
be an $(n-1)$-dimensional Cohen--Macaulay complex 
over $\FF$. Suppose $\Delta$ satisfies the 
Link Condition with respect to an edge $e \in 
\Delta$. If $\cC_\Delta(e)$ is Cohen--Macaulay 
over $\FF$ of dimension $n-1$ and both 
$\link_\Delta(e)$ and $\cC_\Delta(e)$ are strong 
(respectively, almost strong) Lefschetz over 
$\FF$, then so is $\Delta$.
\end{proposition}

We note that since, e.g., by Reisner's criterion 
\cite{Rei76}, the Cohen--Macaulay property is 
inherited by links, it is guaranteed that 
$\link_{\Delta}(e)$ is Cohen--Macaulay. On the 
contrary, the contraction of an edge does not even 
need to be pure. \Cref{prop:LinkCondition} was 
proved in \cite[Proposition 3.2]{Mur10} for the 
strong Lefschetz property if $\FF$ is an arbitrary 
infinite field of any characteristic (see also 
\cite[Theorem 2.2]{BN10} for the same result in 
characteristic zero). Since it is not entirely 
obvious, although reasonable to believe, that 
the proofs go through for the almost strong 
Lefschetz property, we sketch the main steps of 
the proof.

\medskip
\noindent
\emph{Proof of Proposition~\ref{prop:LinkCondition}}.
Let $V$ be the vertex set of $\Delta$ and $e = \{a, 
b\} \in \Delta$, where $a < b$. Following 
\cite{Mur10}, we consider the shift operator 
\begin{align*}
C_e(F) \ = \ \begin{cases}
(F \sm \{b\}) \cup \{a\}, &\mbox{ if } b \in F, \,
a \notin F \mbox{ and } (F \sm \{b\}) \cup \{a\} 
\notin \Delta,\\
F, &\mbox{ otherwise,}
\end{cases}
\end{align*}
which goes back to \cite{EKR61}, and set $\Shift_e
(\Delta) = \{C_e(F): F \in \Delta\}$. Since the Link 
Condition holds for $e$, \cite[Lemma 2.1]{Mur10} 
implies that 
$$
\Shift_e(\Delta) \ = \ \cC_\Delta(e) \cup \{\{b\}
\cup F~:~F \in a \ast \link_\Delta(e)\}.
$$
This implies that $\Shift_e(\Delta) = \cC_\Delta(e) 
\cup \Star_\Delta(e)$ and, as a result, there is the 
exact sequence of $\FF[x_v: v \in V]$-modules
\begin{equation}\label{eq:exact}
0 \to \FF[\Star_\Delta(e)] \to \FF[\Shift_e(\Delta)] 
  \to \FF[\cC_\Delta(e)] \to 0,
\end{equation}
where the first map is given by multiplication with 
$x_b$. Since $\link_\Delta(e)$ is $(n-3)$-Lefschetz, 
so is $\Star_\Delta(e)$ (see, e.g., 
\cite[Lemma 2.1]{KN09}). Hence, there exist $\Theta
= (\theta_1, \theta_2,\dots,\theta_n)$ and a linear 
form $\omega \in \FF[x_v: v \in V]$ such that $\Theta$ 
is an l.s.o.p. for $\FF[\Star_\Delta(e)]$, 
$\FF[\Shift_e(\Delta)]$ and $\FF[\cC_\Delta(e)]$ 
simultaneously and $\omega$ is an $(n-1)$- and 
$(n-3)$-Lefschetz element for $\cC_\Delta(e)$ and 
$\Star_\Delta(e)$, respectively, with respect to 
$\Theta$. Hence, from \eqref{eq:exact} we get the 
commutative diagram 
$$\begin{array}{ccccccccc}
   0&\to& \FF(\Star_\Delta(e))_{\ell-1} &\to& \FF(\Shift_e(\Delta))_\ell &\to & \FF(\cC_\Delta(e))_\ell&\to& 0\\
    & & & & & & &&\\
   & & \downarrow \omega^{n-3-2(\ell-1)}& &\ \ \downarrow \omega^{n-2\ell-1}& & \ \downarrow \omega^{n-2\ell-1}& & \\
     & & & & & & &&\\
      &\empty &  \FF(\Star_\Delta(e))_{n-2-\ell} &\to& \FF(\Shift_e(\Delta))_{n-1-\ell}&\to & \FF(\cC_\Delta(e))_{n-1-\ell}&\to &0
   \end{array}$$  
for $0 \le \ell \leq \lfloor (n-1)/2\rfloor$, where 
we have written $\FF(\Star_\Delta(e))$ for 
$\FF[\Star_\Delta(e)]/\Theta$ and similarly for 
$\FF(\Shift_e(\Delta))$ and $\FF(\cC_\Delta(e))$, 
and we have set $\FF(\Star_\Delta(e))_{-1} = 0$. 

Since the left and right vertical maps are injective 
by assumption, so is the middle map by the snake 
lemma. Thus, $\Shift_e(\Delta)$ has the almost strong 
Lefschetz property. Moreover, since $\Delta$ 
satisfies the Link Condition with respect to $e$, we 
conclude from \cite[Lemma 2.2]{Mur10} that 
$I_{\Shift_e(\Delta)}$ is an initial ideal of 
$I_\Delta$ with respect to a certain term order. 
Finally, $\Delta$ has the $(n-1)$-Lefschetz 
property by \cite[Proposition 2.9]{Wie04}. Wiebe's 
orginal result was for $\mathfrak{m}$-primary 
homogeneous ideals having the strong Lefschetz 
property. However, the same proof works in our 
setting. 
\qed\\

Given a simplicial complex $\Delta$ and 
face $U = \{u_1, u_2,\dots,u_n\} \in \Delta$, we 
say that $\Delta$ satisfies the \emph{strong Link 
Condition} with respect to $U$ if 
\begin{equation}\label{eq:StrongLinkCondition}
\link_\Delta(F)\cap\link_\Delta(G) \ = \ 
\link_\Delta(F\cup G)
\end{equation}
for all $F,G\subseteq U$ with $F \cap G = \varnothing$. 
Note that, in this case, \eqref{eq:StrongLinkCondition} 
holds for all (not necessarily disjoint) subsets $F,
G\subseteq U$. The following technical lemma relates 
the strong Link Condition to the usual Link Condition.

\begin{lemma}\label{lem:LinkCondition}
Let $\Delta$ be a simplicial complex which satisfies 
the strong Link Condition with respect to the face
$U = \{u_1, u_2,\dots,u_n\}$. Then, $\cC_\Delta
(\{u_{n-1},u_n\})$ satisfies the strong Link Condition 
with respect to $U \sm \{u_n\}$. 

In particular, all edges of $2^U$ can be contracted 
successively so that at each step, the Link Condition 
is satisfied with respect to the contracted edge.
\end{lemma}

\begin{proof}
To simplify notation, we set $\Delta' = \cC_\Delta
(\{u_{n-1},u_n\})$ and let $U' = U \sm \{u_n\} \in 
\Delta'$ be the contraction of $U$. We 
consider disjoint sets $F, G \subseteq U'$ and 
observe that, by definition of $\Delta'$, 
\begin{eqnarray}
\link_{\Delta'}(F) & = &
\{H \in \link_\Delta(F)~:~u_n\notin H\} \, \cup
\label{eq:link1} \\
& & \{(H \sm \{u_n\}) \, \cup \,
\{u_{n-1}\}~:~u_n\in H\in\link_\Delta(F),\ u_{n-1} 
\notin H\} \nonumber
\end{eqnarray}
if $u_{n-1} \notin F$, and
\begin{eqnarray} 
\link_{\Delta'}(F) & = &  
\{H \in \link_\Delta(F)~:~u_n\notin H\} \, \cup
\label{eq:link2} \\ 
& & \{ H \sm \{u_{n-1}\}~:~ H \in 
\link_{\Delta}((F \sm \{u_{n-1}\})\cup\{u_n\}) \}
\nonumber \end{eqnarray}     
if $u_{n-1}\in F$. The inclusion 
\[ \link_{\Delta'}(F\cup G) \ \subseteq \ 
   \link_{\Delta'}(F) \cap \link_{\Delta'}(G) \]
holds trivially. To prove the reverse inclusion, we 
consider a face $H \in \link_{\Delta'}(F) \cap 
\link_{\Delta'}(G)$ and distinguish two cases.

{\sf Case 1:} $u_{n-1}\notin F\cup G$. By 
Equation~\eqref{eq:link1} for $\link_{\Delta'}(F)$ 
and $\link_{\Delta'}(G)$, four cases can occur. 
First, assume that $H \in \link_\Delta(F) \cap 
\link_\Delta(G)$. Then, by the strong Link Condition, 
$H\in \link_{\Delta}(F\cup G)$ and hence $H \in 
\link_{\Delta'}(F\cup G)$ by \eqref{eq:link1}. Next, 
suppose that $H = (H' \sm \{u_n\}) 
\cup \{u_{n-1}\}$ for some $H' \in \link_\Delta(F) 
\cap \link_\Delta(G)$ with $u_n \in H'$, $u_{n-1} 
\notin H$. Then, the strong Link Condition implies 
that $H' \in \link_\Delta(F\cup G)$ and hence $H \in 
\link_{\Delta'}(F\cup G)$ by \eqref{eq:link1}. 
Finally, assume that $H \in \link_\Delta(F)$ and 
that $H = (H' \sm \{u_n\}) \cup \{u_{n-1}\}$ for 
some $H' \in \link_\Delta(G)$ with $u_n \in H'$ and 
$u_{n-1} \notin H'$. Then, $F\cup H
\in \Delta$ and $G \cup (H \sm \{u_{n-1}\}) \cup 
\{u_n\} \in \Delta$. From  
the strong Link Condition, we conclude that $H \sm
\{u_{n-1}\} \in \link_\Delta(F\cup G\cup\{u_n\})$, 
i.e., $(H \sm \{u_{n-1}\}) \cup \{u_n\} \in 
\link_{\Delta}(F\cup G)$. This, together with  
\eqref{eq:link1} applied to $\link_{\Delta'}
(F\cup G)$, implies again that $H \in \link_{\Delta'}
(F\cup G)$. The remaining case follows by symmetry 
from the previous one. 

{\sf Case 2:} $u_{n-1}\in F\cup G$. Since $F \cap G
= \varnothing$, we may assume without loss of generality  
that $u_{n-1} \in F$ and $u_{n-1} \notin G$. Since 
$H \in \link_{\Delta'}(F)$ and $u_{n-1} \in F$, we 
must have $u_{n-1} \notin H$. From 
Equation~\eqref{eq:link1}, which applies to 
$\link_{\Delta'}(G)$, and the fact that $u_{n-1} \notin 
H$ we conclude that $H \in \link_\Delta(G)$. Two 
subcases can occur. Suppose first that $H \in 
\link_\Delta(F)$. Then, the strong Link Condition 
implies that $H \in \link_\Delta(F\cup G)$ and thus 
$H \in \link_{\Delta'}(F\cup G)$ by 
Equation~\eqref{eq:link2}, applied to 
$\link_{\Delta'}(F \cup G)$. Otherwise, $H \notin 
\link_\Delta(F)$ and we must have $H \in \link_{\Delta}
((F \sm \{u_{n-1}\}) \cup \{u_n\})$ by \eqref{eq:link2}. 
This implies that $H \cup (F \sm \{u_{n-1}\}) \cup 
\{u_n\} \in \Delta$. Since $H \cup G\in \Delta$, from 
the strong Link Condition we infer that $H \cup (F \sm 
\{u_{n-1}\}) \cup \{u_n\} \cup G \in \Delta$. Since 
$u_{n-1} \in F$, we conclude that $H \cup F \cup G \in 
\Delta'$ and so, once again, $H \in \link_{\Delta'}
(F\cup G)$. This completes the proof of the first 
statement.

For the second statement we note that if $\Delta$ 
satisfies the strong Link Condition with respect to 
$U$, then it also satisfies the Link Condition with 
respect to any edge of $2^U$. Hence, the claim follows 
from successive applications of the first statement. 
\end{proof}

\section{Lefschetz properties of antiprism 
triangulations}
\label{sec:aLefschetz}

This section aims to prove \Cref{thm:LefschetzUnimodal}, 
i.e., to show that the antiprism triangulation of any 
shellable simplicial complex has the almost strong 
Lefschetz property over $\RR$. From this we will 
infer that the $h$-vector of the antiprism 
triangulation of any Cohen--Macaulay simplicial 
complex is unimodal and will locate its peak. 

We first show that the antiprism triangulation of the 
simplex $\sigma_n$ has the almost strong Lefschetz 
property over $\RR$. The next lemma will be crucial. 
Recall that the (strong) Link Condition was defined in
Section~\ref{sec:Lefschetz}.

\begin{lemma}\label{lem:SLCAntiprism}
Consider an $(n-1)$-dimensional simplex $2^V$ and a
triangulation $\Delta$ of its boundary complex 
$\partial(2^V)$. Then, the antiprism 
$\Gamma_\aA(\Delta)$ satisfies the strong Link 
Condition with respect to the set of its interior 
vertices. 

In particular, $\sd_\aA(\sigma_n)$ satisfies the 
strong Link Condition with respect to the set of 
its interior vertices.
\end{lemma}

\begin{proof}
Set $V = \{v_1, v_2,\dots,v_n\}$ and let $U = \{u_1,
u_2,\dots,u_n\}$ be the set of interior vertices of 
$\Gamma_\aA(\Delta)$, linearly ordered so that 
$\{u_i, v_i\} \not\in \Gamma_\aA(\Delta)$ for every 
$i \in [n]$. 

Let $E = \{u_i : i \in I\} \subseteq U$ for some 
$I \subseteq [n]$ be nonempty and let $\bar{E} = 
\{u_j : j \in [n] \sm I\}$ and $\bar{F} = \{v_j : 
j \in [n] \sm I\}$ be the faces of the simplices 
$2^U$ and $2^V$, respectively, which are 
complementary to $E$. Then, by definition of 
$\Gamma_\aA(\Delta)$, 
\begin{equation*}
\link_{\Gamma_\aA(\Delta)}(E) \ = \ \Delta_\aA 
(\Delta_{\bar{F}}),
\end{equation*}
where the new vertices added for the $\Delta_\aA$ 
construction (see Remark~\ref{rem:DeltaA}) are the 
elements of $\bar{E}$ and $\Delta_{\bar{F}}$ is the 
restriction of $\Delta$ to the (proper) face 
$\bar{F} \in 2^V$. This directly implies that 
$\Gamma_\aA(\Delta)$ satisfies the strong Link 
Condition with respect to $U$.
\end{proof}

Given a Cohen--Macaulay simplicial complex $\Delta$ 
over a field $\FF$, we say that the contraction of 
an edge $e \in 
\Delta$ is \emph{admissible over $\FF$} if $\Delta$ 
satisfies the Link Condition with respect to $e$ and 
$\link_{\Delta}(e)$ is strong Lefschetz over $\FF$. 
The following proposition is, essentially, a 
consequence of Lemmas~\ref{lem:LinkCondition} 
and~\ref{lem:SLCAntiprism}.
\begin{proposition} \label{prop:aASimplexToBoundary}
There exists a sequence of admissible edge contractions 
over $\RR$ which transforms $\sd_\aA(\sigma_n)$ into 
the cone over its boundary. In particular, 
$\sd_\aA(\sigma_n)$ is almost strong Lefschetz over 
$\RR$, if $\partial(\sd_\aA(\sigma_{n}))$ is strong 
Lefschetz over $\RR$.
\end{proposition}

\begin{proof}
As before, we let $V=\{v_1,\ldots,v_n\}$ and $U =
\{u_1,\ldots,u_n\}$ be the vertices of $\sigma_n$ and 
the interior vertices of $\sd_\aA(\sigma_n)$, 
respectively. By \Cref{lem:SLCAntiprism}, $\sd_\aA
(\sigma_n)$ satisfies the strong Link Condition with 
respect to $U$. Thus, using \Cref{lem:LinkCondition}, 
we can successively contract edges from $U$, each 
satisfying the Link Condition, until we reach a single 
vertex $u$. The resulting complex is clearly the cone 
$u \ast \partial(\sd_\aA(\sigma_n))$. If we can verify 
that the intermediate complexes, appearing in this 
sequence of contractions, are Cohen--Macaulay over 
$\RR$ and that the links of the contracted edges are 
strong Lefschetz, then \Cref{prop:LinkCondition} 
implies that $\sd_\aA(\sigma_n)$ is almost strong 
Lefschetz, if so is the cone 
$u \ast \partial(\sd_\aA(\sigma_n))$. 

To prove the missing statements, we use the fact that 
$\sd_\aA(\sigma_n)$ can be constructed from $\sigma_n$ 
by crossing operations on its faces, starting at the 
facet $V$ and moving to faces of lower dimension (see 
\Cref{sec:atriang}). From this it follows that the 
intermediate complexes can be constructed by first 
contracting the corresponding edges in the antiprism 
$\Gamma_\aA(\partial(2^V))$ and then performing 
crossing operations on its boundary faces. The 
antiprism $\Gamma_\aA(\partial(2^V))$ is a regular 
triangulation of $2^V$ and so is any subcomplex 
obtained from it by the performed edge contractions. 
Since, in addition, any crossing operation can be 
realized by a sequence of stellar subdivisions 
(see the proof of \cite[Theorem~8]{BM87}), which are 
well known to preserve regularity, we conclude that 
any intermediate complex in the sequence of edge 
contractions from $\sd_\aA(\sigma_n)$ to the cone 
$u \ast \partial(\sd_\aA(\sigma_n))$ is a regular 
triangulation of $2^V$ and, in particular, 
Cohen--Macaulay over $\RR$. Moreover, the regularity 
of the intermediate complexes implies that the link 
of any interior edge that is contracted is a 
polytopal sphere and hence strong Lefschetz over 
$\RR$ \cite{Sta80}. Using  
\Cref{prop:LinkCondition}, we conclude that 
$\sd_\aA(\sigma_n)$ is almost strong Lefschetz over 
$\RR$, if so is $u \ast \partial(\sd_\aA(\sigma_n))$. 
By \cite[Lemma 2.1]{KN09}, this is the case if 
$\partial(\sd_\aA(\sigma_n))$ is strong Lefschetz 
over $\RR$.
\end{proof}

The next statement suffices to conclude that 
$\sd_\aA(\sigma_n)$ has the almost strong Lefschetz 
property over $\RR$.

\begin{proposition}\label{prop:SLPSigma}
The simplicial complex $\partial(\sd_\aA(\sigma_n))$ 
is combinatorially isomorphic to the boundary complex
of a simplicial polytope. In particular, it is strong 
Lefschetz over $\RR$.
\end{proposition}

\begin{proof}
We use again the fact that 
$\partial(\sd_\aA(\sigma_n))$ can be constructed from 
$\partial \sigma_n$ by a sequence of crossing operations 
(see \Cref{sec:atriang}). As already mentioned, it was 
shown in the proof of Theorem 8 in \cite{BM87} that 
every crossing operation can be expressed as a sequence 
of stellar subdivisions. Since those preserve 
polytopality, the first statement follows. The second 
follows from the first and \cite{Sta80}.
\end{proof}

The next result follows by combining 
Propositions~\ref{prop:aASimplexToBoundary} 
and~\ref{prop:SLPSigma}.

\begin{theorem} \label{thm:SimplexLefschetz}
The simplicial complex $\sd_\aA(\sigma_n)$ is almost 
strong Lefschetz over $\RR$.
\end{theorem}

\begin{remark} \rm
Since the restriction of the antiprism triangulation 
$\sd_\aA(\sigma_n)$ to a face $F$ of $\sigma_n$ is 
the antiprism triangulation of $2^F$, we can apply 
the edge contractions from the proof of 
\Cref{prop:aASimplexToBoundary} to the subdivided 
faces of $\partial(\sd_\aA(\sigma_n))$, ordered by 
decreasing dimension. Clearly, the simplicial 
complex obtained in this way is combinatorially 
isomorphic to the barycentric subdivision of 
$\partial\sigma_n$. Using similar arguments as in 
the proof of \Cref{prop:aASimplexToBoundary}, one 
can show that all edge contractions are admissible. 
Indeed, let $\Delta'$ be the simplicial complex 
obtained from $\partial(\sd_\aA(\sigma_n))$ after 
$i$ edge contractions. Consider a face $F \in 
\sigma_n$ and let $\Delta'_F$ be the restriction 
of $\Delta'$ to $F$ (which is the subcomplex of 
$\Delta'$ consisting of all faces with carrier 
contained in $F$). If $G \in \Delta'_F$ is 
a face having all vertices in the interior of 
$\Delta'_F$, then it can easily be verified that 
\begin{equation}\label{eq:link}
\link_{\Delta'}(G) \ = \ \link_{\Delta'_F}(G) 
\ast \{\{u_{H_1},\ldots,u_{H_r}\}~:~F\subsetneq 
H_1\subsetneq \cdots \subsetneq H_r\subsetneq [n]\},
\end{equation}
where $u_H$ denotes the last interior vertex in the 
sequence of contractions of a face $H \in \sigma_n$. 
If $G$ is the edge to be contracted in $\Delta'$, 
then the previous equation, combined with the proof 
of \Cref{prop:aASimplexToBoundary}, implies that 
$\Delta'$ satisfies the Link Condition with respect
to $G$. We further note that the second complex on 
the right-hand side of \eqref{eq:link} is isomorphic 
to the barycentric subdivision of 
$\link_{\partial\sigma_n}(F)$, which itself is the 
barycentric subdivision of the boundary complex of 
an $(n-2-|F|)$-dimensional simplex. Thus, by \cite[Proposition 2.3]{KN09} and the 
proof of Proposition~\ref{prop:aASimplexToBoundary}, both simplicial 
complexes on the right-hand side of \eqref{eq:link} 
are strong Lefschetz over $\RR$. From this fact and
\cite[Theorem 1.2 (i)]{BN10}, it follows that 
$\link_{\Delta'}(G)$ is strong Lefschetz over $\RR$. 

The previous discussion shows that there exists a 
sequence of admissible edge contractions transforming 
$\partial(\sd_\aA(\sigma_n))$ into 
$\partial(\sd(\sigma_n))$. This provides another 
proof of the second statement of \Cref{prop:SLPSigma}. 
Moreover, the analogous edge contraction can be applied 
to the edge links. This shows that the antiprism 
triangulation of $\partial \sigma_n$ is strongly edge 
decomposable (see \cite[Definition 1.1]{Mur10}), since 
so is the barycentric subdivision of $\partial \sigma_n$. 
\qed
\end{remark}

With \Cref{prop:SLPSigma} at hand, the key 
observation to complete the proof of the first 
statement of \Cref{thm:LefschetzUnimodal} is that 
the proof of \cite[Theorem 1.1]{KN09}, showing that 
the barycentric subdivision of any shellable simplicial 
complex is almost strong Lefschetz over an infinite 
field (in particular, over $\RR$), works for every
uniform triangulation which fulfills this property 
for simplices. For the interested reader, and to 
keep this article as self-contained as possible, 
we provide a sketch of the proof.

\begin{theorem} \label{thm:LefschetzShellable}
The complex $\sd_\aA(\Delta)$ is almost strong 
Lefschetz over $\RR$ for every shellable simplicial 
complex $\Delta$.
\end{theorem}

\begin{proof}
Let $\dim(\Delta) = n-1$, as usual. The proof 
proceeds by double induction on $n$ and the number 
of facets of $\Delta$. At the  
base of the induction, either $\Delta$ consists 
only of vertices, in which case there is nothing 
to show, or $\Delta$ is a simplex, in which case 
the result follows from \Cref{prop:SLPSigma}.

For the inductive step we assume that $n \ge 2$, 
let $V$ be the vertex set of $\Delta$ and let $A 
= \RR[x_v : v\in V]$. Consider a shelling $G_1, 
G_2,\dots,G_m = G$ of $\Delta$ and set 
$\widetilde{\Delta}:=\langle G_1,\dots,G_{m-1}
\rangle$ and $\tau := \widetilde{\Delta}\cap 2^G$. 
There is the following exact sequence of $A$-modules:
\begin{equation} \label{eq:MV-seq}
0 \rightarrow \RR[\sd_\aA(\Delta)] \rightarrow 
\RR[\sd_\aA(\widetilde{\Delta})] \oplus 
\RR[\sd_\aA(2^G)] \rightarrow 
\RR[\sd_\aA(\tau)]\rightarrow 0.
\end{equation}
%

One now chooses generic linear forms $\Theta = 
\theta_1, \theta_2,\dots,\theta_n$ so that $\Theta$ 
is an l.s.o.p. for $\RR[\sd_\aA(\Delta)]$, 
$\RR[\sd_\aA(\widetilde{\Delta})]$ and 
$\RR[\sd_\aA(2^G)]$ simultaneously and $\theta_1, 
\theta_2,\dots,\theta_{n-1}$ is an l.s.o.p. for 
$\RR[\sd_\aA(\tau)]$. Dividing out by $\Theta$ in 
\eqref{eq:MV-seq} gives rise to an exact sequence
\begin{eqnarray*}
\Tor_1(\RR[\sd_\aA(\tau)],A/\Theta)\stackrel{\delta}{\rightarrow} \RR(\sd_\aA(\Delta))\rightarrow \RR(\sd_\aA(\widetilde{\Delta}))\oplus \RR(\sd_\aA(2^G))\rightarrow \RR(\sd_\aA(\tau))\rightarrow 0,
\end{eqnarray*}
where we have written $\RR(\sd_\aA(\Delta))$ for 
$\RR[\sd_\aA(\Delta)]/\Theta$ and similarly for 
$\RR(\sd_\aA(\widetilde{\Delta}))$, 
$\RR(\sd_\aA(\tau))$ and $\RR(\sd_\aA(2^G))$. Next, 
one shows 
that $\Tor_1(\RR[\sd_\aA(\tau)],A/\Theta)_i=0$ for 
$0 \leq i \leq \lfloor \frac{n-2}{2}\rfloor$. This 
is done exactly as in the proof of 
\cite[Theorem 1.1]{KN09}. Since all maps in the 
previous exact sequence preserve the grading, one 
gets the commutative diagram
$$\begin{array}{ccccccc}
    \RR(\sd_\aA(\Delta))_i &\to& \RR(\sd_\aA(\widetilde{\Delta}))_i\oplus \RR(\sd_\aA(2^G))_i\\
     & & & & & &\\
     \downarrow \omega^{n-2i-1}& &\ \ \downarrow (\omega^{n-2i-1},\omega^{n-2i-1})& & \\
      & & & & & &\\
       \RR(\sd_\aA(\Delta))_{n-1-i} &\to& \RR(\sd_\aA(\widetilde{\Delta}))_{n-1-i}\oplus \RR(\sd_\aA(2^G))_{n-1-i}
   \end{array}$$   
where $\omega$ is a degree one element in $A$. The 
induction hypothesis implies that the multiplication 
map on the right-hand side is injective for a 
generic $\omega$. One concludes that the 
multiplication
  \begin{equation*} 
	\omega^{n-2i-1}:\hspace{5pt} \RR(\sd_\aA(\Delta))_i\rightarrow \RR(\sd_\aA(\Delta))_{n-1-i}
	\end{equation*}
is injective for $0 \leq i\leq \lfloor\frac{n-2}{2}
\rfloor$ and the proof follows.
\end{proof}

To complete the proof of \Cref{thm:LefschetzUnimodal}, 
we need the following properties of the numbers 
$p_\aA(n,k,j)$, discussed in Section~\ref{sec:transform}.

\begin{lemma}\label{lemma:propertiesTrafo}
Let $n,k,j\in \NN$ with $k,j\leq n$. 
\begin{itemize}
\item[(a)] $p_\aA(n,k,j)=p_\aA(n,n-k,n-j)$,
\item[(b)] \begin{equation*}
p_\aA(n,k,0) \, \le \, p_\aA(n,k,1) \, \le \cdots \le 
\, p_\aA(n,k,\lfloor n/2\rfloor)
\end{equation*}
and
\begin{equation*}
p_\aA(n,k,n) \, \le \, p_\aA(n,k,n-1) \, \le \cdots 
\le \, p_\aA(n,k,\lceil n/2\rceil).
\end{equation*}
\end{itemize}
\end{lemma}

\begin{proof}
Part (a) follows from 
\cite[Proposition 4.6 (a)]{Ath20+}; we will 
provide a direct bijective proof.
Let $Q(n,k,j)$ be defined as in the proof of 
\Cref{prop:h-transform}, so that $p_\aA(n,k,j) = 
|Q(n,k,j)|$. Also, let $\widetilde{Q}(n,n-k,n-j)$ 
be the set of triples defining $Q(n,n-k,n-j)$, 
except that the set $S$ which is partitioned 
satisfies $\{k+1,\dots,n\} \subseteq S \subseteq [n]$,
instead of $[n-k] \subseteq S \subseteq [n]$, and 
that the monochromatic block, if present, must be
contained in $\{k+1,\ldots,n\}$. Since, clearly, 
$|\widetilde{Q}(n,n-k,n-j)| = p_\aA(n,n-k,n-j)$, 
to prove (a) it suffices to find a 
bijection from $Q(n,k,j)$ to $\widetilde{Q}
(n,n-k,n-j)$. Given a triple in $Q(n,k,j)$, 
consisting of an ordered partition of $[k] 
\subseteq S \subseteq [n]$ and a suitable coloring of 
the elements of $S$, we construct a triple in 
$\widetilde{Q}(n,n-k,n-j)$ as follows. We first 
switch the colors of all elements of $S$ from white 
to black and vice versa. If the first block was 
monochromatic, we delete it from the partition. The 
block $[n] \sm S$, if nonempty, is then added to 
the constructed ordered partition as its new first 
block, with all its elements colored black. We leave 
to the reader to verify that this process gives a well 
defined map. The inverse map can be constructed by 
the same procedure, applied to the triples in 
$\widetilde{Q}(n,n-k,n-j)$.

For part (b), we note that the proof of 
\cite[Corollary 4.4]{KN09} works, with the symmetry 
of \cite[Lemma 2.5]{BW08} replaced by that of part (a).
\end{proof}

We recall that a sequence $(a_0, a_1,\dots,a_s) \in 
\NN^{s+1}$ is an \emph{M-sequence} if it is the 
Hilbert function of a standard graded $\FF$-algebra, 
where $\FF$ is a field. Macaulay \cite{Mac27} provided 
a characterization of M-sequences by means of 
numerical conditions (see, e.g., \cite{StaCCA}).
\Cref{thm:LefschetzShellable} has the following 
numerical consequences for the $h$-vector of the 
antiprism triangulation of a Cohen--Macaulay complex.

\begin{corollary}
Let $\Delta$ be a Cohen--Macaulay simplicial complex 
of dimension $n-1$ and let $g(\sd_\aA(\Delta)) = 
(1, h_1(\sd_\aA(\Delta)) - 
h_0(\sd_\aA(\Delta)),\dots,h_{\lfloor n/2\rfloor}
(\sd_\aA(\Delta)) - h_{\lfloor n/2\rfloor - 1}
(\sd_\aA(\Delta)))$.  
\begin{itemize}
\item[(a)] $g(\sd_\aA(\Delta))$ is an M-sequence.
\item[(b)] $h(\sd_\aA(\Delta))$ is unimodal. The peak 
is at position $n/2$, if $n$ is even, and at position 
$(n-1)/2$ or $(n+1)/2$, if $n$ is odd.
\item[(c)] $h_i(\sd_\aA(\Delta))\leq h_{n-1-i}
(\sd_\aA(\Delta))$ for all $0 \le i \le 
\lfloor(n-2)/2\rfloor$.
\end{itemize}
\end{corollary}

\begin{proof}
Parts (a) and (c) follow by standard arguments used 
when one works with Lefschetz properties, see, e.g., 
\cite[Sections 3.1--3.2]{HMMNWW13}.
For part (b), the proof of \cite[Corollary 4.7]{KN09} 
works, with \Cref{lemma:propertiesTrafo} replacing 
\cite[Corollary 4.4 (ii)]{KN09}.
\end{proof}

We conclude this section by recording the following 
properties of the Stanley-Reisner ring of the antiprism 
triangulation of any simplicial complex.

\begin{proposition}
Let $\Delta$ be a simplicial complex. 
\begin{itemize}
\item[(a)] $\dim (\FF[\sd_\aA(\Delta)]) = 
\dim(\FF[\Delta])$.
\item[(b)] $\depth (\FF[\sd_\aA(\Delta)]) = \depth 
(\FF[\Delta])$.
\item[(c)] 
\begin{equation*}\reg(\FF[\sd_\aA(\Delta)]) \ = \
\begin{cases}
\dim(\Delta), \quad \quad \mbox{if } 
\widetilde{H}_{\dim(\Delta)}(\Delta;\FF) = 0,\\
\dim(\Delta)+1, \quad \mbox{if } 
\widetilde{H}_{\dim(\Delta)}(\Delta;\FF) \neq 0.
\end{cases}
\end{equation*}
\end{itemize}
\end{proposition}

\begin{proof}
Part (a) is clear, since $\dim\sd_\aA(\Delta) = 
\dim(\Delta)$. Part (b) follows from \cite[Theorem
3.1]{Mun84}, since $\Delta$ and $\sd_\aA(\Delta)$ 
have homeomorphic geometric realizations. Part (c)
follows from an application of Hochster's formula 
\cite[Corollary~5.12]{MS05}; one can also mimick 
the detailed argument for the barycentric subdivision 
given in the proof of \cite[Proposition 2.6]{KW08}.
%
\end{proof}

\section{Further directions}
\label{sec:direct}

This section concludes with comments, open problems 
and further directions for research.

\medskip
\noindent

1. The question asking which uniform triangulations 
transform $h$-polynomials with nonnegative coefficients 
into polynomials with only real roots was raised in 
\cite{Ath20+}. This property has been verified for 
several examples, including the prototypical one
of the barycentric subdivision \cite{BW08}, and is 
conjectured in this paper for the antiprism 
triangulation. We believe that this property is not 
uncommon among uniform triangulations. 

For instance, consider any word $w = w_1 w_2 \cdots 
w_d$ with $w _i \in \{a, b\}$ for every $i \in [d]$ 
and let $\Delta$ be a $d$-dimensional simplicial 
complex. Let $\sd_w(\Delta)$ be the triangulation 
of $\Delta$ defined inductively as follows. Assume 
that all faces of $\Delta$ of dimension less than 
$j$ have been triangulated, for some $j \in [d]$. 
Then, triangulate each $j$-dimensional face $F \in
\Delta$ by the antiprism construction, if $w_j = 
a$, and by the coning construction, if $w_j = b$,
over the already triangulated boundary of $F$. By 
applying this process for $j = 1, 2,\dots,d$, in
this order, we get a triangulation $\sd_w(\Delta)$ 
of $\Delta$ which coincides with $\sd_\aA(\Delta)$, 
when $w = aa \cdots a$, and with $\sd(\Delta)$, 
when $w = bb \cdots b$. It seems plausible that 
this triangulation has the same property for every 
$w$, but it is not easy to deduce such a statement
from the results of \cite{Ath20+} and this paper.
Example~\ref{ex:binom-euler} corresponds to the 
word $w = bb \cdots ba$. 

\medskip
\noindent
2. The symmetric polynomials $\ell_\aA(\sigma_n, x)$ 
and $\bar{p}_\aA(\sigma_n, x)$ were shown and 
conjectured, respectively, to be 
$\gamma$-positive in Section~\ref{sec:enumerate}. It
is an interesting (and possibly challenging) open 
problem to find explicit combinatorial interpretations 
of the corresponding $\gamma$-coefficients. Similar 
remarks apply to the symmetric polynomials 
$h_\aA(\partial \sigma_n, x)$ and 
$\theta_\aA(\sigma_n, x) = h_\aA(\sigma_n, x) - 
h_\aA(\partial \sigma_n, x)$, shown to be 
real-rooted, and hence $\gamma$-positive, in 
Section~\ref{sec:enumerate}. 

\medskip
\noindent
3. The local $h$-polynomial of the barycentric 
subdivision of any CW-regular subdivision of the 
simplex was shown to be $\gamma$-positive 
in~\cite{KMS19}. Does this hold if `barycentric 
subdivision' and `CW-regular subdivision' are 
replaced by `antiprism triangulation' and 
`CW-regular simplicial subdivision', 
respectively?

\medskip
\noindent
4. The $h$-polynomial of the barycentric subdivision 
of any doubly Cohen--Macaulay simplicial complex 
and the barycentric subdivision of any triangulation 
of a ball were shown to have a nonnegative real-rooted 
symmetric decomposition in~\cite[Section~5]{BS20} 
and~\cite[Section~8]{Ath20+}, respectively. Do these 
statements hold if `barycentric subdivision' is 
replaced by `antiprism triangulation'?

\medskip
\noindent
\emph{Note added in revision}. After this paper was 
written, the almost strong Lefschetz property was 
established for a very broad class of 
triangulations of Cohen--Macaulay polyhedral complexes 
(including antiprism triangulations of Cohen--Macaulay
simplicial complexes) by Adiprasito 
and Yashfe~\cite[Theorem~46]{AY20} using completely 
different methods and tools, including the partition 
complex.


\bigskip
\noindent \textbf{Acknowledgments}. Part of this 
research was conducted while the first author was 
in residence at the Institute Mittag-Leffler in 
Djursholm, Sweden during the `Algebraic and 
Enumerative Combinatorics' program in Spring 2020, 
supported by the Swedish Research Council under 
grant no. 2016-06596. The second and third author 
were partially supported by the German Research 
Council DFG GRK-1916. The third author is grateful 
to Satoshi Murai for pointing out that his original 
proof of \cite[Proposition 3.2]{Mur10} can be made 
less technical by using the exact sequence 
\eqref{eq:exact}.

Data sharing not applicable to this article as no
datasets were generated or analysed during the 
current study.

\end{document}